\documentclass[11pt,a4paper]{article}

\usepackage{cmap}
\usepackage[T1]{fontenc}

\usepackage{amsmath,amssymb,latexsym,graphicx,mathrsfs,amsthm}
\usepackage{secdot}
\sectiondot{subsection}
\usepackage{a4wide}
\usepackage{enumitem}

\DeclareRobustCommand\nlab{r}

\usepackage{xcolor}
\usepackage{url}
\usepackage{hyperref}
\definecolor{ForestGreen}{rgb}{0.1,0.6,0.05}
\definecolor{EgyptBlue}{rgb}{0.063,0.1,0.6}
\definecolor{RipeOlive}{HTML}{556B2F}
\hypersetup{
	colorlinks=true,
	linkcolor=EgyptBlue,         
	citecolor=ForestGreen,
	urlcolor=RipeOlive
}

\usepackage[hyperpageref]{backref}
\usepackage{epstopdf}
\newtheorem{theorem}{Theorem}
\newtheorem{proposition}[theorem]{Proposition}
\newtheorem{lemma}[theorem]{Lemma}

\theoremstyle{definition}

\newtheorem{remark}[theorem]{Remark}

\numberwithin{equation}{section}
\numberwithin{theorem}{section}

\numberwithin{equation}{section}
\numberwithin{theorem}{section}

\newenvironment{proof*}[1]{\begin{trivlist}\item[\hskip%
		\labelsep{{\bf Proof of \/{\rm\bf #1.}}\quad}]\rm}%
	{\hfill\qed\rm\end{trivlist}}

\newcommand{\W}{W_0^{1,p}}

\newcommand{\intO}{\int_\Omega}
\newcommand{\C}{C^1_0(\overline{\Omega})}
\newcommand{\E}{E_{\alpha,\beta}}
\newcommand{\En}{E_{\alpha_n,\beta}}

\title{
	\vspace*{-1cm}
	On the Fredholm-type theorems and sign properties of solutions for $(p,q)$-Laplace equations with two parameters}
\author{ 
	\normalsize Vladimir Bobkov\\ 
	{\small  Department of Mathematics and NTIS, Faculty of Applied Sciences, University of West Bohemia}\\ 
	{\small Univerzitn\'i 8, 306 14 Plze\v{n}, Czech Republic}\\
	{\small e-mail: bobkov@kma.zcu.cz}\\[0.5em] 
	\normalsize Mieko Tanaka\\
	{\small Department of  Mathematics, 
		Tokyo University of Science}\\
	{\small Kagurazaka 1-3, Shinjyuku-ku, Tokyo 162-8601, Japan}\\
	{\small e-mail: miekotanaka@rs.tus.ac.jp} 
}

\date{}

\begin{document}
\maketitle 
	\begin{abstract} 
		We consider the Dirichlet problem for the nonhomogeneous equation $-\Delta_p u -\Delta_q u = \alpha |u|^{p-2}u + \beta |u|^{q-2}u + f(x)$ in a bounded domain, where $p \neq q$, and $\alpha, \beta \in \mathbb{R}$ are parameters.
		We explore assumptions on $\alpha$ and $\beta$ that guarantee the resolvability of the considered problem. Moreover, we introduce several curves on the $(\alpha,\beta)$-plane allocating sets of parameters for which the problem has or does not have positive or sign-changing solutions, provided $f$ is of a constant sign.
			
	
	\end{abstract}


\section{Introduction and preliminaries}\label{sec:intro}
Consider the boundary value problem
\begin{equation}\label{eq:D}
\tag{$D_{\alpha,\beta,f}$}
\left\{
\begin{aligned}
-\Delta_p u -\Delta_q u &= \alpha |u|^{p-2}u+\beta |u|^{q-2}u + f(x)
&&\text{in}\ \Omega, \\
u&=0 &&\text{on }\ \partial \Omega,
\end{aligned}
\right.
\end{equation}
where $\Delta_r u := \text{div}\left(|\nabla u|^{r-2} \nabla u \right)$ with 
$r=p$ or $r=q$ defines the $r$-Laplace operator, $p,q>1$ and, without loss of generality, $p>q$. 
Parameters $\alpha, \beta$ are real numbers, and $\Omega\subset \mathbb{R}^N$ is a bounded domain with $C^2$-boundary, $N \geq 1$.
The source function $f$ belongs to $W^{-1,p'}(\Omega)$, the dual of the Sobolev space $W_0^{1,p} := W_0^{1,p}(\Omega)$, $p' = \frac{p}{p-1}$. The latter space is endowed with the norm $\| \nabla \left(\cdot\right) \|_p$, where $\|u\|_p := \left(\intO |u|^p \, dx\right)^{1/p}$ defines the norm of $L^p(\Omega)$.

The main ``building block'' of \eqref{eq:D} is the nonlinear eigenvalue problem for the $r$-Laplacian
\begin{equation}\label{eq:Egen}
\tag*{(\theequation)$_{\nlab}$}
\left\{
\begin{aligned}
-\Delta_r u &= \lambda |u|^{r-2}u
&&\text{in}\ \Omega, \\
u&=0 &&\text{on }\ \partial \Omega.
\end{aligned}
\right.
\end{equation}
We say that $\lambda$ is an eigenvalue of the $r$-Laplacian if there exists a nonzero weak solution of \ref{eq:Egen} called an eigenfunction associated with $\lambda$. 
Hereinafter, we denote by $\lambda_1(r)$ and $\varphi_r$ the first eigenvalue and an associated eigenfunction, respectively. It is known (see, e.g., \cite{anane1987,Alleg}) that $\lambda_1(r)$ is simple and isolated, and it can be defined by
\begin{equation}\label{eq:lambdar}
\lambda_1(r) 
= 
\inf\left\{\frac{\|\nabla u\|_r^r}{\|u\|_r^r}:~ u \in W_0^{1,r} \setminus \{0\}
\right\}.
\end{equation}
Moreover, $\varphi_r \in C^{1,\gamma}(\overline{\Omega})$ for some $\gamma \in (0,1)$ (see Remark \ref{rem:reg} below), and we can choose $\varphi_r$ to be strictly positive \cite{puser}.
By analogy with the linear case $r=2$, we will denote the set of all eigenvalues of the $r$-Laplacian as $\sigma(-\Delta_r)$ and the set of all eigenfunctions associated to the parameter $\lambda \in \mathbb{R}$ as $ES(r;\lambda)$. For instance, $ES(r;\lambda_1(r)) \equiv \mathbb{R}\varphi_r$, and $ES(r;\lambda) = \{0\}$ provided $\lambda \not\in \sigma(-\Delta_r)$. 

\smallskip
If we let $p=q$ and $\alpha=\beta$, then, up to scaling, \eqref{eq:D} turns to the problem
\begin{equation}\label{eq:E}
\left\{
\begin{aligned}
-\Delta_p u &= \lambda |u|^{p-2}u + f(x)
&&\text{in}\ \Omega, \\
u&=0 &&\text{on }\ \partial \Omega,
\end{aligned}
\right.
\end{equation}
which had been actively studied in the following two directions:
\begin{itemize}
	\item[1.] \textit{Existence and multiplicity} of solutions. 
	In fact, if $\lambda$ is not an eigenvalue of $-\Delta_p$, then the existence is well-known, see, e.g., \cite[Theorem 3.1, p.\ 60]{FNSS}. Moreover, in the linear case $p=2$ the complete information about the existence and multiplicity is provided via the Fredholm alternative. 
	On the other hand, the situation is drastically different in the nonlinear case $p \neq 2$. 
	It was investigated in \cite{drabek,DGTU,takac2} (see also  references therein), that the nontrivial multiplicity of solutions of \eqref{eq:E} can occur as for $\lambda = \lambda_1(p)$, and for $\lambda$ from a punctured neighborhood of $\lambda_1(p)$, depending on assumptions on $f$. 
	Notice that the investigations of \eqref{eq:E} were carried out mainly in a neighborhood of $\lambda_1(p)$ due to the lack of description of $\sigma(-\Delta_p)$ at higher eigenvalues.
	\item[2.] \textit{Sign properties} of solutions. 
	Suppose that $f \in L^\infty(\Omega) \setminus \{0\}$ and $f \geq 0$ a.e.\ in $\Omega$.
	Then the well-known maximum principle states that any solution of \eqref{eq:E} is positive provided $\lambda < \lambda_1(p)$. Moreover, any solution of \eqref{eq:E} is either nonnegative or sign-changing for $\lambda > \lambda_1(p)$, see, e.g., \cite[Theorem 2.1]{Alleg}. 
	The latter result is strengthened by the anti-maximum principle \cite{CP,FGTT}: there exists $\lambda_f > \lambda_1(p)$ such that any solution of \eqref{eq:E} is negative provided $\lambda \in (\lambda_1(p), \lambda_f)$.
\end{itemize}

The aim of the present paper is to obtain some basic results on the existence of solutions of \eqref{eq:D}, as well as on their sign properties.
As in the case of the problem \eqref{eq:E}, it is clear that the investigation of \eqref{eq:D} should be preceded by the study of the corresponding unperturbed problem 
\begin{equation}\label{eq:D0}
\tag{$D_{\alpha,\beta,0}$}
\left\{
\begin{aligned}
-\Delta_p u -\Delta_q u &= \alpha |u|^{p-2}u+\beta |u|^{q-2}u
&&\text{in}\ \Omega, \\
u&=0 &&\text{on }\ \partial \Omega.
\end{aligned}
\right.
\end{equation}
While the left- and right-hand sides of \eqref{eq:D0} have the same ``homogeneity'', the structure of the solution set of \eqref{eq:D0} strongly depends on the choice of the parameters $\alpha$ and $\beta$, see, e.g., \cite{BobkovTanaka2015,BobkovTanaka2016,BobkovTanaka2017,MT,T-2014}, where the existence, multiplicity, and behavior of solutions of \eqref{eq:D0} have been comprehensively studied. 
In fact, depending on $\alpha$ and $\beta$, \eqref{eq:D0} demonstrates a behavior similar to the one for problems with convex-concave nonlinearities \cite{abc} and indefinite nonlinearities \cite{altar,ilyas}, but with a more essential inclination to eigenvalue problems.

Let us recall several results from \cite{BobkovTanaka2015,BobkovTanaka2017} about the existence and nonexistence of positive solutions of \eqref{eq:D0}, some of which will be used in the subsequent sections. 
In \cite{BobkovTanaka2015}, the following family of critical points was introduced:
\begin{equation}\label{eq:beta_ps}
\beta_{ps}(\alpha):=
\sup_{u \in \text{int}\,\C_+} 
\inf_{\varphi \in C^1_0(\overline{\Omega})_+ \setminus\{0\}}
\mathcal{L}_\alpha(u; \varphi)
\quad 
\text{ for } 
\alpha \geq \lambda_1(p).
\end{equation}
Here $\mathcal{L}_\alpha(u; \varphi)$ is the so-called \textit{extended functional} (see \cite{ilfunc}) defined as
$$
\mathcal{L}_\alpha(u; \varphi) := 
\frac{\intO |\nabla u|^{p-2}\nabla u\nabla\varphi \,dx +
	\intO |\nabla u|^{q-2}\nabla u\nabla\varphi \,dx-
	\alpha\intO |u|^{p-2} u \varphi\,dx}{\intO |u|^{q-2} u \varphi\,dx}, 
$$
and $\text{int}\,\C_+$ is the interior of the positive cone $\C_+$ of $C^1_0(\overline{\Omega})$ given by
$$
\text{int}\,\C_+ := \left\{u \in C^1_0(\overline{\Omega}):~ 
u(x)>0 \text{ for all } x \in \Omega,~
\frac{\partial u}{\partial\nu}(x) < 0 \text{ for all } x \in \partial\Omega \right\},
$$
where $\nu$ is the outward unit normal vector to $\partial\Omega$. 
It was proved in \cite[Theorem 2.2]{BobkovTanaka2015} that $\beta_{ps}(\alpha)$ is the threshold curve on the $(\alpha,\beta)$-plane which separates sets of the existence and nonexistence of positive solutions of \eqref{eq:D0}. 
Namely, if $\alpha > \lambda_1(p)$, then \eqref{eq:D0} has a positive solution when $\beta < \beta_{ps}(\alpha)$, and \eqref{eq:D0} does not have positive solutions when $\beta > \beta_{ps}(\alpha)$. Moreover, if $\alpha < \lambda_1(p)$, then \eqref{eq:D0} possesses a solution if and only if $\beta > \lambda_1(q)$, see \cite[Propositions 1 and 2]{BobkovTanaka2015}. 
Borderline cases were also studied, see, in particular, \cite[Proposition 4]{BobkovTanaka2015} and \cite[Proposition 3]{BobkovTanaka2017}.
As for the properties of the curve $\beta_{ps}(\alpha)$, it is known that $\beta_{ps}(\lambda_1(p)) \geq \beta_*$, $\beta_{ps}(\alpha)$ is nonincreasing, and $\beta_{ps}(\alpha) = \lambda_1(q)$ for all $\alpha \geq \alpha_*$, where
\begin{equation}\label{eq:alpha_*}
\alpha_* := \frac{\|\nabla \varphi_q\|_p^p}{\|\varphi_q\|_p^p},
\quad
\beta_* := \frac{\|\nabla \varphi_p\|_q^q}{\|\varphi_p\|_q^q},
\end{equation}
see \cite[Proposition 3]{BobkovTanaka2015} and \cite[Proposition 3 (ii)]{BobkovTanaka2017}.
Notice that $\alpha_* > \lambda_1(p)$ and $\beta_* > \lambda_1(q)$, since $\lambda_1(p)$ and $\lambda_1(q)$ are simple and have different eigenspaces, see \cite[Proposition 13]{BobkovTanaka2017}.

We also refer the interested reader to the works \cite{averna,chaves,flip,marcomasconi} for existence results and qualitative properties of solutions of other types of problems driven by the $(p,q)$-Laplace operator.


\section{Main results}\label{sec:mainresults}
In this section, we collect our main results. We group them according to the existence of solutions of \eqref{eq:D} and their qualitative properties.

Hereinafter, we will use the notations
\begin{equation*}\label{def:HG}
H_\alpha (u) := \|\nabla u\|_p^p -\alpha\|u\|_p^p 
\quad \text{and}\quad 
G_\beta(u) := \|\nabla u\|_q^q -\beta\|u\|_q^q,
\end{equation*}
and
$$
\langle f,u\rangle := \intO fu\,dx
$$
for the dual pairing between  $f\in W^{-1,p'}(\Omega)$ and $u \in \W$.

\subsection{Existence}
\begin{theorem}\label{thm:existence} 
Let $f\in W^{-1,p'}(\Omega)$. 
Assume that $\alpha,\beta \in \mathbb{R}$ are such that one of the following assumptions is satisfied:
\begin{enumerate}[label={\rm(\roman*)}]
	\item\label{thm:existence:G-} $G_\beta (u)<0$ for all $u\in ES(p;\alpha)\setminus\{0\}$;
	\item\label{thm:existence:G+} $G_\beta (u)>0$ for all $u\in ES(p;\alpha)\setminus\{0\}$.
\end{enumerate}
Then \eqref{eq:D} has at least one solution.
\end{theorem} 

In order to quantify the assumptions \ref{thm:existence:G-} and \ref{thm:existence:G+} of Theorem \ref{thm:existence}, 
we introduce the following two families of critical values: 
\begin{align}
\underline{\beta}(\alpha) 
&:=
\inf\left\{\frac{\|\nabla u\|_q^q}{\|u\|_q^q}:
~ u \in ES(p;\alpha)\setminus\{0\}\right\}, 
\label{def:beta2}\\
\bar{\beta}(\alpha)  
&:=
\sup\left\{\frac{\|\nabla u\|_q^q}{\|u\|_q^q}:
~ u \in ES(p;\alpha)\setminus\{0\}\right\}, 
\quad \alpha \in \mathbb{R},
\label{def:beta1}
\end{align} 
and set $\underline{\beta}(\alpha)=+\infty$ and $\bar{\beta}(\alpha)=-\infty$ for $\alpha\not\in\sigma(-\Delta_p)$. 
Note that $\underline{\beta}(\alpha), \bar{\beta}(\alpha) \in [\lambda_1(q),+\infty)$ provided $\alpha\in\sigma(-\Delta_p)$, see \cite[Lemma 3.6]{BobkovTanaka2016}, and the simplicity of $\lambda_1(p)$ and $\lambda_1(q)$ yields
\begin{equation}\label{def:beta_*}
\underline{\beta}(\lambda_1(p)) = \bar{\beta}(\lambda_1(p)) = \beta_* > \lambda_1(q), 
\end{equation}
where $\beta_*$ is defined in \eqref{eq:alpha_*} and the inequality follows from \cite[Proposition 13]{BobkovTanaka2017}.
Then Theorem \ref{thm:existence} can be reformulated as follows.
\begin{theorem}\label{cor:existence} 
Let $f\in W^{-1,p'}(\Omega)$. Assume that $\alpha,\beta \in \mathbb{R}$ are such that either  $\beta<\underline{\beta}(\alpha)$ or $\beta>\bar{\beta}(\alpha)$. Then \eqref{eq:D} has at least one solution.
\end{theorem} 

The equalities in \eqref{def:beta_*} show that if $\alpha=\lambda_1(p)$, then Theorem \ref{thm:existence} does not provide the existence for \eqref{eq:D} only when $\beta=\beta_*$. We have the following result in this case.
\begin{theorem}\label{thm:existence2}
	Let $\partial \Omega$ be connected if $N \geq 2$.
	Assume that $p > 2q$, $\alpha = \lambda_1(p)$ and $\beta = \beta_*$.  
	If $f\in L^2(\Omega)$ is such that $\intO f \varphi_p\, dx = 0$, then \eqref{eq:D} has at least one solution.
\end{theorem}

Recalling $p>q$, let us remark that the term $-\Delta_q u - \beta |u|^{q-2} u$ in the problem \eqref{eq:D} can be considered as a specific case of a lower order perturbation of the $p$-Laplacian, and some results stated in this section have to be typical for a more general settings, too. 
In particular, the $(p,q)$-Laplacian can be seen as a $(p-1)$-quasihomogeneous operator in the sense of \cite[Definition 2.1, p.~58]{FNSS}, and hence \cite[Theorem 3.2, p.~73]{FNSS} implies the existence of a solution of \eqref{eq:D} whenever $f \in W^{-1,p'}(\Omega)$ and $\alpha \not\in \sigma(-\Delta_p)$. Thus, our Theorems \ref{thm:existence} and \ref{thm:existence2} provide an improvement of this existence result.

\subsection{Sign properties}\label{subsec:sign}

Let us start with the following two remarks about a regularity of solutions of \eqref{eq:D} and further properties of nonnegative solutions of \eqref{eq:D}.
\begin{remark}\label{rem:reg}
	Assume that $\{\alpha_n\}$ and $\{\beta_n\}$ are bounded, $\{f_n\} \subset L^\infty(\Omega)$ is such that $\{\|f_n\|_\infty\}$ is bounded, and $\{c_n\}$ is nonnegative and bounded. 
	If $u_n$ is a (weak) solution of 
	\begin{equation*}
	\left\{
	\begin{aligned}
	-\Delta_p u - c_n \Delta_q u &= \alpha_n |u|^{p-2}u + \beta_n |u|^{q-2}u + f_n(x)
	&&\text{in}\ \Omega, \\
	u&=0 &&\text{on }\ \partial \Omega,
	\end{aligned}
	\right.
	\end{equation*}
	and $\{u_n\}$ is bounded in $\W$, then there exist $\gamma \in (0,1)$ and $M > 0$ independent of $n \in \mathbb{N}$ such that $u_n\in C^{1,\gamma}(\overline{\Omega})$ and $\|u_n\|_{C^{1,\gamma}(\overline{\Omega})} \leq M$ 
	for all $n$.
	Indeed, thanks to the boundedness assumptions, Moser's iteration process (see, e.g., \cite[Appendix C]{MMT}) implies that $\{u_n\}$ is bounded in $L^\infty(\Omega)$, and then regularity results \cite{Lieberman} and \cite[Theorem 1.7]{L} yield the boundedness of $\{u_n\}$ in  $C^{1,\gamma}(\overline{\Omega})$. 
	In particular, if $f\in L^\infty(\Omega)$, then any solution of \eqref{eq:D} belongs to $C^{1,\gamma}(\overline{\Omega})$ for some $\gamma \in (0,1)$. 
	The same regularity holds true for any eigenfunction of the $r$-Laplacian, $r>1$. 
\end{remark}

\begin{remark}\label{ref:sign}
	If $f \in L^\infty(\Omega)$, $f \geq 0$ a.e.\ in $\Omega$, and $u \geq 0$ is a nonzero nonnegative solution of \eqref{eq:D}, then $u > 0$ in view of the maximum principle \cite[Theorem 5.4.1]{puser}. Moreover, Remark \ref{rem:reg} and the boundary point lemma \cite[Theorem 5.5.1]{puser} yield $u \in \text{int}\,\C_+$. Analogously, $\varphi_r \in \text{int}\,\C_+$, where $\varphi_r$ is the first (nonnegative) eigenfunction of the $r$-Laplacian, $r>1$.
\end{remark}

Assuming $f \in W^{-1,p'}(\Omega)$ and $f \geq 0$ in the weak sense, we introduce the following family of critical values:
\begin{equation}
\beta_f(\alpha) := \inf \left\{
\Phi_\alpha^+(u):~ u\in\mathcal{B}^+(\alpha) \right\},
\quad
\alpha \in \mathbb{R},
\label{beta_f} 
\end{equation}
where
\begin{align} 
\Phi_\alpha^+(u) 
&:=\frac{\|\nabla u\|_q^q}{\|u\|_q^q} + \frac{p-1}{p-q}\left(\frac{p-q}{q-1}\right)^\frac{q-1}{p-1}\frac{\left(H_\alpha(u)\right)^\frac{q-1}{p-1}
\langle f,u\rangle^\frac{p-q}{p-1}}{\|u\|_q^q}, 
\label{def:Phi} \\ 
\mathcal{B}^+(\alpha) &:=\left\{u \in \W \setminus \{0\}:~ 
u \geq 0 \text{ a.e.\ in } \Omega \text{ and } H_\alpha(u) \geq 0\right\}. 
\label{def:B^+} 
\end{align}
Clearly, $\beta_f(\alpha) \in [\lambda_1(q),+\infty)$ for any $\alpha \in \mathbb{R}$. 
In Lemma \ref{lem:prop:beta_f} below we study some other properties of $\beta_f$. In particular, we show that $\beta_f(\alpha) > \lambda_1(q)$ if and only if $\alpha < \alpha_*$, provided $\langle f, \varphi_q \rangle > 0$, where $\alpha_*$ is defined in \eqref{eq:alpha_*}.

\begin{theorem}\label{thm:nodal}
	Let $f \in W^{-1,p'}(\Omega) \setminus \{0\}$ and $f \geq 0$ in the weak sense. 
	If $\alpha\le\lambda_1(p)$ and $\beta < \beta_f(\alpha)$, then 
	any solution of \eqref{eq:D} is nonnegative. 
	If, in addition, $f \in L^\infty(\Omega)$, then 
	any solution of \eqref{eq:D} belongs to $\mathrm{int}\,\C_+$.
\end{theorem} 

The result of Theorem \ref{thm:nodal} can be complemented by the following dichotomy. 
\begin{proposition}\label{prop:nodal2}
	Let $f \in L^\infty(\Omega) \setminus \{0\}$ and $f \geq 0$ a.e.\ in $\Omega$. 
	Then for every $\beta < \beta_f(\lambda_1(p))$ there exists $\delta(\beta)>0$ such that for all $\alpha \in (\lambda_1(p), \lambda_1(p) + \delta(\beta))$ any solution of \eqref{eq:D} is either positive or negative. In particular, \eqref{eq:D} has no sign-changing solutions. 
\end{proposition}

By means of the critical curve $\beta_{ps}(\alpha)$ defined in \eqref{eq:beta_ps}, we give the following fact. 
\begin{proposition}\label{thm:ATM} 
Let $f\in L^\infty(\Omega) \setminus \{0\}$ and $f \geq 0$ a.e.\ in $\Omega$. 
Assume that $\alpha \ge \lambda_1(p)$ and $\beta>\beta_{ps}(\alpha)$. 
Then \eqref{eq:D} has no nonnegative solutions. 
That is, any solution of \eqref{eq:D} is either nonpositive or sign-changing.
\end{proposition}

Recall that $\beta_{ps}(\alpha)=\lambda_1(q)$ for all $\alpha \geq \alpha_*$. In this case, Proposition \ref{thm:ATM} can be refined as follows.
\begin{proposition}\label{thm:ATM2} 
	Let $f\in L^\infty(\Omega) \setminus \{0\}$ and $f \geq 0$ a.e.\ in $\Omega$. 
	Then for any $\alpha > \alpha_*$ there exists $\varepsilon(\alpha)>0$ such that for any $\beta \geq \lambda_1(q) - \varepsilon(\alpha)$, \eqref{eq:D} has no nonnegative solutions.
\end{proposition}

We schematically depict the results of Theorems \ref{cor:existence}, \ref{thm:nodal}, and Propositions \ref{thm:ATM}, \ref{thm:ATM2} in Figure \ref{fig:1} below.
\begin{figure}[h!]
	\center{
		\includegraphics[width=0.7\linewidth]{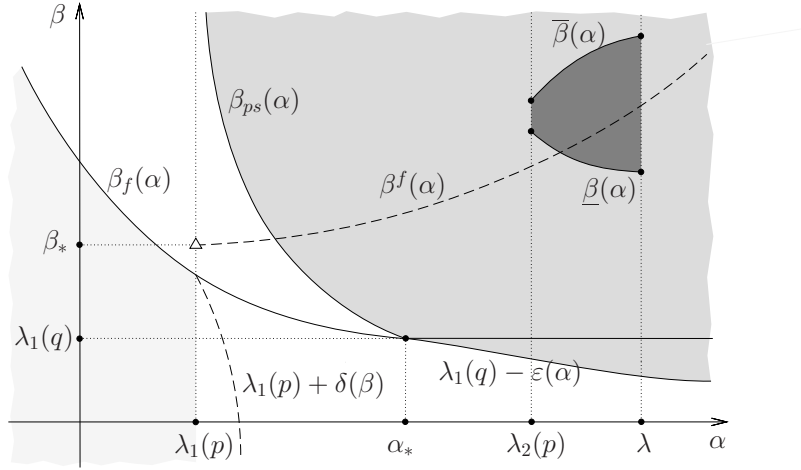}
		\caption{Assume that $f \in L^\infty(\Omega) \setminus \{0\}$, $f \geq 0$ a.e.\ in $\Omega$, and $(\lambda_2(p), \lambda)\subset \sigma(-\Delta_p)$ for some $\lambda > \lambda_2(p)$. Light gray - any solution is positive. Gray - any solution is either nonpositive or sign-changing. Dark gray - existence is unknown.}
		\label{fig:1}
	}
\end{figure}

To obtain additional qualitative properties of solutions of \eqref{eq:D}, we introduce one another family of critical values:
\begin{equation}\label{beta^f}
\beta^f(\alpha) := \sup 
\left\{
\Phi_\alpha^-(u):~ u\in \mathcal{B}^-(\alpha)
\right\}, 
\quad
\alpha \geq \lambda_1(p),
\end{equation}
where
\begin{align}
\Phi_\alpha^-(u) 
&:=\frac{\|\nabla u\|_q^q}{\|u\|_q^q} - \frac{p-1}{p-q}\left(\frac{p-q}{q-1}\right)^\frac{q-1}{p-1}\frac{\left(-H_\alpha(u)\right)^\frac{q-1}{p-1}
\langle f,u\rangle^\frac{p-q}{p-1}}{\|u\|_q^q}, 
\label{def:Phi^-} \\ 
\mathcal{B}^-(\alpha) 
&:=
\left\{
u \in \W \setminus \{0\}:~ 
u \geq 0 \text{ a.e.\ in } \Omega \text{ and }  H_\alpha(u) \leq 0
\right\}. 
\label{def:B^-} 
\end{align}
We show in Lemma \ref{lem:prop:beta^f} below that $\beta^f(\alpha) < +\infty$ for all $\alpha \geq \lambda_1(p)$. 
Note that any function $u \in \W$ can be decomposed as $u = u^+ + u^-$, where $u^+ := \max\{u, 0\} \in \W$ and $u^- := \min\{u, 0\} \in \W$.
\begin{proposition}\label{prop:aux5}
	Let $f \in W^{-1,p'}(\Omega) \setminus \{0\}$ and $f \geq 0$ in the weak sense.
	Assume that $\alpha > \lambda_1(p)$, $\beta \in \mathbb{R}$, and let $u$ be a solution of \eqref{eq:D}. 
	Then the following assertions are satisfied:
	\begin{enumerate}[label={\rm(\roman*)}]
	\item\label{prop:aux5:1} if $\beta < \beta_f(\alpha)$ and $u^- \not\equiv 0$, then $H_\alpha(u^-) < 0$;
	\item\label{prop:aux5:2} if $\beta > \beta^f(\alpha)$ and $u^+ \not\equiv 0$, then $H_\alpha(u^+) > 0$.
	\end{enumerate}
\end{proposition}

\medskip
The rest of the article is organized as follows. 
In Section \ref{sec:proofs_existence}, we prove Theorem \ref{thm:existence}. 
In Section \ref{sec:proofs_existence2}, we prove Theorem \ref{thm:existence2}. 
Section \ref{sec:proofs_sign} is devoted to the proof of the results stated in Section \ref{subsec:sign}.

\section{Proofs. Existence I}\label{sec:proofs_existence}
In this section, we prove Theorem~\ref{thm:existence}.
We start by preparing several auxiliary results.  
We will use the sequence of eigenvalues $\{\lambda_k(p)\} \subset \sigma(-\Delta_p)$ introduced in \cite{DR} which can be defined as
\begin{equation}\label{lambda_n}
\lambda_k(p) := \inf_{h\in\mathscr{F}_k(p)} \max_{x \in S^{k-1}} 
\|\nabla h(x)\|_p^p, 
\end{equation}
where $S^{k-1}$ is the unit sphere in $\mathbb{R}^k$, $k \in \mathbb{N}$, and
\begin{align} 
\label{F_n} 
\mathscr{F}_k(p)
&:=\left\{ h\in C(S^{k-1},S(p)):~ h \text{ is odd}\right\},
\\
\notag
S(p) 
&:=\left\{u\in W_0^{1,p}:~ \|u\|_p=1\right\}.
\end{align} 
It is known that $\lambda_k(p) \to +\infty$ as $k \to +\infty$, see \cite[p.\ 195]{DR}.
However, we recall that it is an open problem whether $\{\lambda_k(p)\} = \sigma(-\Delta_p)$, except in the cases $p=2$ and $N=1$, where the answer is affirmative. 

\smallskip
Along this section we assume that $f \in W^{-1,p'}(\Omega)$, and we denote by $\|f\|_*$ the norm of $f$ in $W^{-1,p'}(\Omega)$.
Recall that weak solutions of \eqref{eq:D} are critical points of the energy functional $\E \in C^1(\W, \mathbb{R})$ defined by
\begin{equation*}\label{def:E} 
\E(u) = \frac{1}{p}\, H_\alpha (u)+\frac{1}{q}\,G_\beta(u) - \langle f,u\rangle. 
\end{equation*}
To prove Theorem \ref{thm:existence}, we show that $\E$ has a linking structure provided $\alpha > \lambda_1(p)$ and $\alpha \neq \lambda_k(p)$, $k \in \mathbb{N}$. 
Then we can obtain a critical point of $\E$ whenever $\E$ satisfies the Palais--Smale condition. 

Let us consider the set
\begin{equation}\label{def:Y} 
Y(\lambda)
:=
\left\{
u\in\W:~ \|\nabla u\|_p^p \geq \lambda\|u\|_p^p 
\right\},
\quad 
\lambda \in \mathbb{R}.
\end{equation}
Hereinafter, $S^k_+$ stands for the closed unit upper hemisphere in $\mathbb{R}^{k+1}$ with the boundary $S^{k-1}$. 
We start by formulating the following linking lemma. 
\begin{lemma}[\protect{\cite[Lemma 3.1]{BobkovTanaka2016}}]\label{lem:link} 
	Let	$k \in \mathbb{N}$.
	Then $h(S^k_+)\cap Y(\lambda_{k+1}(p)) \neq \emptyset$ for any $h\in C(S^k_+,W_0^{1,p})$ provided $h\big|_{S^{k-1}}$ is odd.
\end{lemma}

\begin{lemma}\label{lem:EY}
	Let $\alpha, \beta \in \mathbb{R}$. If $\lambda>\max\{\alpha,0\}$, then $\E$ is bounded from below and coercive on $Y(\lambda)$. 
\end{lemma} 
\begin{proof}
	Let $u \in Y(\lambda)$. 
	We have
	$$
	\E(u) \geq \frac{1}{p}\left(1-\frac{\alpha}{\lambda}\right) \|\nabla u\|_p^p - \frac{\beta}{q(\lambda_1(p))^{q/p}}|\Omega|^\frac{p-q}{p} \|\nabla u\|_p^q - \|f\|_* \|\nabla u\|_p,
	$$
	where we used the H\"older and Poincar\'e inequalities to estimate the term $\|u\|_q^q$ of $\E(u)$. Since $p>q>1$ and $\lambda>\max\{\alpha,0\}$, we easily deduce the desired boundedness from below and coercivity of $\E$ on $Y(\lambda)$.
\end{proof}

\begin{lemma}\label{PS} 
	Let $\alpha,\beta \in \mathbb{R}$ be such that
	\begin{equation}\label{hyp:G} 
	G_\beta (u) \not=0 
	\quad 
	\text{for any } 
	u\in ES(p;\alpha)\setminus\{0\}.
	\end{equation}
	Then $\E$ satisfies the Palais--Smale condition. 
\end{lemma}
\begin{proof} 
	Let $\{u_n\} \subset \W$ be a Palais--Smale sequence for $\E$. 
	In view of the $(S_+)$-property of the operator $-\Delta_p - \Delta_q$ (see, e.g., \cite[Remark 3.5]{BobkovTanaka2016}), it is sufficient to show the boundedness of $\{u_n\}$ in order to establish the desired Palais--Smale condition for $\E$.
	Suppose, by contradiction, that $\|\nabla u_n\|_p \to +\infty$ as $n \to +\infty$, up to a subsequence. Then, arguing in the same way as in \protect{\cite[Lemma 3.3]{BobkovTanaka2016}}, we see that the sequence of normalized functions
	$v_n:=u_n/\|\nabla u_n\|_p$ converges strongly in $\W$ (up to a subsequence) to some 
	$v_0\in ES(p;\alpha)\setminus\{0\}$. 
	Hence, we get a contradiction whenever $\alpha\not\in\sigma(-\Delta_p)$. 
	On the other hand, if $\alpha\in\sigma(-\Delta_p)$, then
	\begin{align*} 
	o(1) =\frac{1}{\|\nabla u_n\|_p^q}
	\left(p\E(u_n) -\left< \E^\prime (u_n),u_n\right>\right) =\left(\frac{p}{q}-1\right)\,G_\beta (v_n)-\frac{p-1}{\|\nabla u_n\|_p^{q-1}}\,
\langle f,v_n\rangle
	\end{align*}
	as $n\to +\infty$. This yields $G_\beta (v_0)=0$, which contradicts \eqref{hyp:G}.
\end{proof} 

\smallskip
We are now in a position to prove Theorem \ref{thm:existence}. The proof will be split into three cases, each of which is considered in a separate subsection.

\subsection{Case \texorpdfstring{$\alpha\not\in \{\lambda_k(p):\,k\in\mathbb{N}\}$}{alpha not lambda-k}}
\label{subsec:nonresonant} 
First, we briefly handle the case $\alpha<\lambda_1(p)$. 
Since $Y(\lambda_1(p)) = \W$, where $Y(\lambda_{1}(p))$ is given by \eqref{def:Y}, Lemma \ref{lem:EY} implies that $\E$ is bounded from below and coercive on $\W$. 
Therefore, there exists a global minimizer of $\E$ which is a solution of \eqref{eq:D}.

Assume now that $\alpha > \lambda_1(p)$ and $\alpha\not\in \{\lambda_k(p):\,k\in\mathbb{N}\}$. (Recall that the case $\alpha \in \sigma(-\Delta_p)$ is still possible.)
Since $\lambda_k(p)\to +\infty$ as $k \to +\infty$, we can find $k \in \mathbb{N}$ such that $\lambda_k(p)<\alpha<\lambda_{k+1}(p)$.
Define 
\begin{gather*} 
\gamma :=\inf\left\{\E(u):~u\in Y(\lambda_{k+1}(p))\right\}>-\infty, 
\\
\Lambda :=\left\{h\in C (S_+^k,\W):~ \max_{x\in S^{k-1}}\E(h(x))
\le \gamma-1\ \text{and}\ h\big|_{S^{k-1}} \ \text{is\ odd}\,\right\}, 
\\
c := \inf_{h\in \Lambda}\max_{x\in S_+^k} \E(h(x)), 
\end{gather*}
where $\gamma$ is well-defined by Lemma \ref{lem:EY}. 
Note that if $\Lambda \neq \emptyset$, then Lemma \ref{lem:link} implies $c\ge \gamma$. 
Moreover, under assumption \ref{thm:existence:G-} or \ref{thm:existence:G+} of the theorem, $\E$ satisfies the Palais--Smale condition by Lemma~\ref{PS}.
Therefore, by the standard arguments via the deformation lemma we see that $c$ is a critical value of $\E$ provided $\Lambda \neq \emptyset$.
Thus, let us construct an admissible map to show that $\Lambda\not=\emptyset$. 

Let us choose some $0<\varepsilon<(\alpha-\lambda_k(p))/2$. 
Then, due to the definition of $\lambda_k(p)$, there exists 
$h_0\in \mathscr{F}_k(p)$ such that 
$$
\max_{x\in S^{k-1}}\|\nabla h_0(x)\|_p^p<\lambda_k(p)+\varepsilon. 
$$
This implies
\begin{equation}\label{eq:H<0}
\max_{x\in S^{k-1}} H_\alpha(h_0(x))<
\lambda_k(p)+\varepsilon-\alpha 
<-\varepsilon.
\end{equation}
Therefore, recalling that $p>q>1$, we can find a sufficiently large $T_0>0$ such that
\begin{align}
\label{eq:E<g-1}
\max_{x\in S^{k-1}} &\E(T h_0(x))
<-\frac{T^p \varepsilon}{p}
+\frac{T^q}{q} \max_{x\in S^{k-1}}G_\beta(h_0(x)) + T \|f\|_*  (\lambda_k(p)+\varepsilon)^{1/p}
\le \gamma -1
\end{align}
for all $T \geq T_0$.
Using Dugundji's extension theorem \cite{dug}, we may assume that $h_0 \in C(S^k_+,\W)$, and hence $T_0 h_0 \in \Lambda$. Thus, $\Lambda \neq \emptyset$, which implies that $c$ is a critical value of $\E$.
\qed

\subsection{Case \texorpdfstring{$\alpha\in \{\lambda_k(p):\,k\in\mathbb{N}\}$}{alpha in lambda-k i} under assumption \ref{thm:existence:G-}}\label{sec:G-} 

We may assume that $\lambda_k(p)=\alpha<\lambda_{k+1}(p)$ 
for some $k\in\mathbb{N}$. 
Let $\{\alpha_n\} \subset \mathbb{R}$ be a decreasing sequence satisfying 
\begin{equation*} 
\lambda_k(p)=\alpha<\alpha_n<\lambda_{k+1}(p) \quad 
\text{for all } 
n\in\mathbb{N},
\text{ and } \lim_{n\to+\infty}\alpha_n=\alpha. 
\end{equation*}
Similarly to Section \ref{subsec:nonresonant} above, consider for each $n \in \mathbb{N}$,
\begin{gather} 
\gamma_n :=\inf\{\En(u):~ u\in Y(\lambda_{k+1}(p))\}>-\infty, 
\nonumber \\
\Lambda_n :=\left\{h\in C (S_+^k,\W):~ \max_{x\in S^{k-1}}\En(h(x))
\le \gamma_n-1\ \text{and}\ h\big|_{S^{k-1}} \ \text{is\ odd}\,\right\}, 
\nonumber \\
c_n := \inf_{h\in \Lambda_n}\max_{x\in S_+^k} \En(h(x)). 
\label{def:cn}
\end{gather}
Arguing as in Section \ref{subsec:nonresonant}, we get $\Lambda_n \neq \emptyset$, and hence Lemma~\ref{lem:link} yields $c_n\ge \gamma_n$. Moreover, since $\{\alpha_n\}$ is a decreasing sequence, we have 
$H_{\alpha_n}(u)\ge H_{\alpha_1}(u)$ for any $n\in\mathbb{N}$ and 
$u\in\W$, whence 
\begin{equation}
\label{eq:cn>gn}
c_n \ge \gamma_n \ge 
\inf\{E_{\alpha_1,\beta}(u):~ u\in Y(\lambda_{k+1}(p))\}>-\infty,
\end{equation}
where the last inequality follows from Lemma \ref{lem:EY}. 
That is, $\{c_n\}$ is bounded from below. 

Now we claim that for any $n \in \mathbb{N}$ and $\varepsilon>0$ there exists 
$u_n^\varepsilon\in\W$ such that 
\begin{equation}
\label{eq:findps}
|\En(u_n^\varepsilon)-c_n|< \varepsilon 
\quad \text{and}\quad 
\|\En^\prime(u_n^\varepsilon)\|_*<\varepsilon. 
\end{equation}
Suppose, by contradiction, that there exist some $n\in\mathbb{N}$ and $\varepsilon>0$ such that 
$$
u\in \En^{-1}([c_n-\varepsilon,c_n+\varepsilon]))
\quad \text{implies} \quad 
\|\En^\prime(u)\|_*\ge \varepsilon.
$$
(Clearly, the set $\En^{-1}([c_n-\varepsilon,c_n+\varepsilon])$ is nonempty.)
Let us fix some $\varepsilon^\prime \in (0, \min\{\varepsilon,c_n-(\gamma_n-1)\})$. 
Then, using the deformation lemma (see, e.g., \cite[Theorem 3.4 of Chapter I]{chang}), we can find a map $\eta\in C(\W,\W)$ which satisfies the following two assertions:
\begin{enumerate}[label={\rm(\roman*)}]
\item\label{proof:thm:1} $\eta(u)=u$ for any $u$ such that 
$u\not\in\En^{-1}([c_n-\varepsilon^\prime,c_n+\varepsilon^\prime])$;
\item\label{proof:thm:2} $\En(\eta(u))\le c_n-\varepsilon^\prime/2$ 
whenever $\En(u)\le c_n+\varepsilon^\prime/2$.
\end{enumerate}
By the definition \eqref{def:cn} of $c_n$, there exists $h\in\Lambda_n$ such that 
$\max\limits_{x\in S_+^k} \En(h(x))\le c_n+\varepsilon^\prime/2$. 
Since $h\in\Lambda_n$, we have $\max\limits_{x\in S^{k-1}} \En(h(x)) \le \gamma_n-1$. Thus, in view of the inequality $\gamma_n-1<c_n-\varepsilon^\prime$, assertion \ref{proof:thm:1} implies that $\eta(h(x))=h(x)$ for all $x\in S^{k-1}$, and hence $\eta\circ h\in \Lambda_n$. 
On the other hand, assertion \ref{proof:thm:2} yields $\max\limits_{x\in S_+^k} \En(\eta(h(x)))\le c_n-\varepsilon^\prime/2$, which contradicts the definition of $c_n$. 
Consequently, the existence of $\{u_n^\varepsilon\} \subset \W$ satisfying \eqref{eq:findps} is shown.

In view of \eqref{eq:findps}, we can find for each $n \in \mathbb{N}$ a function $u_n \in \W$ such that
\begin{equation}
\label{eq:e<1/n}
|\En(u_n)-c_n| < \frac{1}{n} 
\quad \text{and}\quad 
\|\En^\prime(u_n)\|_*<\frac{1}{n}.
\end{equation}
If $\|\nabla u_n\|_p \to +\infty$ as $n \to +\infty$, up to a subsequence, 
then by the same arguments as in the proof of \cite[Lemma 3.3]{BobkovTanaka2016} the second inequality in \eqref{eq:e<1/n} implies that the sequence of normalized functions $v_n:=u_n/\|\nabla u_n\|_p$ has a subsequence strongly convergent in $\W$ to some $v_0\in ES(p;\alpha)\setminus\{0\}$. 
Moreover, passing to the limit along this subsequence in 
\begin{align*} 
\frac{p(c_n-1)}{\|\nabla u_n\|_p^q} 
-\frac{1}{n\|\nabla u_n\|_p^{q-1}}
&\le \frac{1}{\|\nabla u_n\|_p^q}
\left(\,p\En(u_n) - \left<\En^\prime (u_n),u_n\right> \right) 
\\
&=\left(\frac{p}{q}-1\right)\,G_\beta (v_n)
-\frac{p-1}{\|\nabla u_n\|_p^{q-1}}\,\langle f,v_n\rangle
\end{align*}
and recalling that $\{c_n\}$ is bounded from below (see \eqref{eq:cn>gn}), 
we get $G_\beta(v_0) \geq 0$. 
However, this contradicts assumption \ref{thm:existence:G-} of the theorem. 
Therefore, $\{u_n\}$ is bounded in $\W$ and, consequently, $\{c_n\}$ is also bounded.
Furthermore, noting that
\begin{align*}
\|\E^\prime(u_n)\|_*
\le \|\E^\prime(u_n)-\En^\prime(u_n)\|_*
+\|\En^\prime(u_n)\|_*
\leq 
\frac{\alpha_n-\alpha}{\lambda_1(p)^{1/p}}\|u_n\|_p^{p-1}
+\frac{1}{n},
\end{align*}
we obtain from \eqref{eq:e<1/n} that $\{u_n\}$ is a (bounded) Palais--Smale sequence for $\E$. 
Consequently, $\E$ possesses a critical point since it satisfies the Palais--Smale condition by Lemma \ref{PS}.

\subsection{Case \texorpdfstring{$\alpha\in \{\lambda_k(p):\,k\in\mathbb{N}\}$}{alpha in lambda-k ii} under assumption \ref{thm:existence:G+}}
\label{sec:G+}  
We may assume that $\lambda_k(p)<\alpha=\lambda_{k+1}(p)$ 
for some $k\in\mathbb{N}$. 
Let $\{\alpha_n\}$ be an increasing sequence satisfying 
\begin{equation*}\label{set:seq} 
\lambda_k(p)<\alpha_n<\alpha=\lambda_{k+1}(p) \quad 
\text{for all } 
n\in\mathbb{N},
\text{ and } 
\lim_{n \to +\infty}\alpha_n=\alpha. 
\end{equation*}
Let us define $\gamma_n$, $\Lambda_n$, and $c_n$ as in Section \ref{sec:G-} above. 
By the same arguments as in Section \ref{sec:G-}, 
for any $n \in \mathbb{N}$ we can find $u_n \in \W$ such that 
$$
|\En(u_n)-c_n| < \frac{1}{n} 
\quad \text{and}\quad 
\|\En^\prime(u_n)\|_* < \frac{1}{n}.	
$$
We claim that $\{c_n\}$ is bounded from above. 
If the claim is true, then, as in Section \ref{sec:G-}, the inequality 
\begin{align*} 
\frac{p(c_n+1)}{\|\nabla u_n\|_p^q}+\frac{1}{n\|\nabla u_n\|_p^{q-1}} 
&\ge \frac{1}{\|\nabla u_n\|_p^q}
\left(p\En(u_n) -\left<\En^\prime (u_n),u_n\right>\right) 
\\
&=\left(\frac{p}{q}-1\right)\,G_\beta(v_n)-\frac{p-1}{\|\nabla u_n\|_p^{q-1}}\,\langle f,v_n\rangle,
\end{align*}
in combination with assumption \ref{thm:existence:G+} of the theorem, implies the boundedness of $\{u_n\}$ in $\W$, which yields the existence of a critical point of $\E$.

Let us prove that $\{c_n\}$ is bounded from above.
Note that $\Lambda_n \neq \emptyset$ for any $n \in \mathbb{N}$ by the same arguments as in Section \ref{subsec:nonresonant}. 
In particular, for $n=1$ we can find a map $h_1\in C(S^k_+,\W)$ and a sufficiently large $T_1>0$
such that $h_1\big|_{S^{k-1}}$ is odd, and
\begin{align}
\label{eq:T_1}
&\max_{x\in S^{k-1}} E_{\alpha_1,\beta}(Th_1(x)) 
\leq \gamma_1 - 1
\quad 
\text{for any } T\ge T_1, \\
\label{eq:T_12}
&\max_{x\in S^{k-1}} E_{\alpha_1,\beta}(Th_1(x)) \to -\infty
\quad
\text{as}
\quad
T \to +\infty,
\end{align}
see \eqref{eq:E<g-1}.
Moreover, recalling that $\{\alpha_n\}$ is increasing, we have $H_{\alpha_n}(u)\le H_{\alpha_1}(u)$ and hence $E_{\alpha_n,\beta}(u)\le E_{\alpha_1,\beta}(u)$ for all $n\in\mathbb{N}$ and $u\in\W$. 
Therefore, in view of \eqref{eq:T_12}, for any $n \in \mathbb{N}$ there exists a sufficiently large $T_n > 0$ such that 
\begin{equation}
\label{eq:en<g-1} 
\max_{x\in S^{k-1}}\En(T_n h_1(x)) \le 
\max_{x\in S^{k-1}}E_{\alpha_1,\beta}(T_n h_1(x)) 
\le \gamma_n -1.
\end{equation}
We may assume that $\{T_n\}$ is increasing.

Define a map $h_n \in C(S^k_+,\W)$ for $n\ge 2$ by 
$$
h_n(x):= 
\begin{cases} 
\big(\,2x_{k+1}T_1+T_n(1-2x_{k+1})\,\big)\,
h_1\left(\frac{x^\prime}{\sqrt{1-x_{k+1}^2}},0\,\right) 
& \text{if}\ 0\le x_{k+1}\le \frac{1}{2}, 
\\ 
T_1 h_1\left(\frac{2}{\sqrt{3}}x^\prime,\frac{2}{\sqrt{3}}
\,\sqrt{x_{k+1}^2-\frac{1}{4}}\right) 
& \text{if}\ \frac{1}{2}\le x_{k+1}\le 1, 
\end{cases}
$$
where $x=(x_1,\cdots,x_{k+1})\in S^k_+$ and
$x^\prime=(x_1,\cdots,x_k)$. 
Considering $x_{k+1}=0$, we deduce from \eqref{eq:en<g-1} that $h_n\in \Lambda_n$.
Moreover, if $0\le x_{k+1}\le \frac{1}{2}$, then
$$
\En(h_n(x)) \le E_{\alpha_1,\beta}(h_n(x))
\le \max_{T_1\le T\le T_n} \max_{x\in S^{k-1}}E_{\alpha_1,\beta}(T h_1(x)) \leq \gamma_1-1
$$
by \eqref{eq:T_1}, and 
if $\frac{1}{2} \le x_{k+1}\le 1$, then
$$
\En(h_n(x)) \leq \max_{x\in S^{k}_+}E_{\alpha_n,\beta}(T_1 h_1(x)) \leq \max_{x\in S^{k}_+}E_{\alpha_1,\beta}(T_1 h_1(x)) < +\infty.
$$
Therefore, since $c_n \leq \max\limits_{x\in S_+^k} \En(h_n(x))$ for each $n \in \mathbb{N}$, we conclude that $\{c_n\}$ is bounded from above, which finishes the proof of Theorem \ref{thm:existence}.

\section{Proofs. Existence II}\label{sec:proofs_existence2}
In this section, we prove Theorem~\ref{thm:existence2}.
First, we provide the following auxiliary result.  
Let us decompose any $u \in L^2(\Omega)$ as $u = \gamma_u \varphi_p + u^\perp$, where $\varphi_p$ is the first eigenfunction of the $p$-Laplacian, 
\begin{equation}
\label{eq:l2dec}
\gamma_u := \|\varphi_p\|_2^{-2} \intO u \varphi_p \, dx
\quad\text{and}\quad
\intO u^\perp  \varphi_p \, dx = 0.
\end{equation}
It is clear that $u^\perp=u-\gamma_u\varphi_p\in\W$ provided 
$u\in\W$. 

\begin{lemma}\label{lem:estimate:G}
	Let $p \geq 2$. Assume that $\beta \in (\lambda_1(q), \beta_*]$. 
	Then there exists a constant $C>0$ such that for any $u \in \W$ satisfying $G_\beta(u) < 0$ there holds
	\begin{equation}
	\label{eq:gest}
	|G_\beta(u)| 
	\leq 
	C |\gamma_u|^{q-1} \left(\intO |\nabla \varphi_p|^{p-2} |\nabla u^\perp|^2 \, dx\right)^\frac{1}{2} 
	+ 
	C \intO |\nabla u^\perp|^{q} \, dx.
	\end{equation}
	If, in addition, $p \geq 2q$ and $\gamma_u \neq 0$, then the following inequality is satisfied:
	\begin{equation}
	\label{eq:gest1}
	|G_\beta(u)| 
	\leq 
	C |\gamma_u|^{-\frac{p-2q}{2}} \left(|\gamma_u|^{p-2}\intO |\nabla \varphi_p|^{p-2} |\nabla u^\perp|^2 \, dx + \intO |\nabla u^\perp|^{p} \, dx \right)^\frac{1}{2}
	+
	C \intO |u^\perp|^{q} \, dx.
	\end{equation}
\end{lemma}
\begin{proof}
	Let us fix any $u = \gamma_u \varphi_p + u^\perp \in \W$.
	Using the mean value theorem, we can find $\varepsilon \in (0,1)$ such that
	\begin{align}
	\notag
	0 > G_\beta(u) 
	&= 
	|\gamma_u|^{q} G_\beta(\varphi_p) + \langle G_\beta'(\gamma_u \varphi_p+\varepsilon u^\perp), u^\perp \rangle
	\geq
	\langle G_\beta'(\gamma_u \varphi_p+\varepsilon u^\perp), u^\perp \rangle \\
	\label{eq:estim:Inf:G}
	&\geq - q \intO |\nabla(\gamma_u \varphi_p+\varepsilon u^\perp)|^{q-1}|\nabla u^\perp| \, dx - 
	\beta q \intO |\gamma_u \varphi_p+\varepsilon u^\perp|^{q-1}|u^\perp| \, dx,
	\end{align}
	where $G_\beta(\varphi_p) \geq 0$ follows from the assumption $\beta \in (\lambda_1(q), \beta_*]$. 
	
	During the proof, we will denote by $C>0$ various constants 
independent of $u$. 
	To estimate \eqref{eq:estim:Inf:G}, we develop an approach from the proof of \cite[Proposition 11]{BobkovTanaka2017}.
	Let us start with the first summand in \eqref{eq:estim:Inf:G}. 
	Since $\varepsilon \in (0,1)$, we have
	$$
	\intO |\nabla(\gamma_u \varphi_p+\varepsilon u^\perp)|^{q-1}|\nabla u^\perp| \, dx 
	\leq 
	\intO \left(|\gamma_u| |\nabla \varphi_p|+ |\nabla u^\perp|\right)^{q-1}|\nabla u^\perp| \, dx.
	$$
	Then, using H\"older's inequality, we get
	\begin{align}
	\notag
	\intO &\left(|\gamma_u| |\nabla \varphi_p|+ |\nabla u^\perp|\right)^{q-1}|\nabla u^\perp| \, dx 
	\leq 
	C |\gamma_u|^{q-1} \intO |\nabla \varphi_p|^{q-1} |\nabla u^\perp| \, dx
	+
	C\intO |\nabla u^\perp|^q \, dx\\
	\notag
	&=	
	C|\gamma_u|^{q-1}\intO |\nabla \varphi_p|^\frac{p-2}{2} |\nabla u^\perp| \cdot |\nabla \varphi_p|^\frac{2q-p}{2} \, dx
	+
	C\intO |\nabla u^\perp|^q \, dx
	\\
	\label{eq:g4est}
	&\leq
	C |\gamma_u|^{q-1}\left(\intO |\nabla \varphi_p|^{p-2} |\nabla u^\perp|^2 \, dx\right)^\frac{1}{2} \left(\intO \frac{dx}{|\nabla \varphi_p|^{p-2q}}\right)^\frac{1}{2}
	+
	C\intO |\nabla u^\perp|^q \, dx.
	\end{align}
	Note that $\intO |\nabla \varphi_p|^{p-2} |\nabla u^\perp|^2 \, dx < +\infty$, since $\varphi_p \in C^{1}(\overline{\Omega})$ (see Remark \ref{rem:reg}) and $p \geq 2$. 
	Moreover, $\intO \frac{dx}{|\nabla \varphi_p|^{p-2q}} < +\infty$, too. Indeed, if $p \leq 2q$, then this boundedness easily follows from the regularity $\varphi_p \in C^{1}(\overline{\Omega})$, while in the case $p > 2q$ the desired boundedness was discussed in \cite[p.\ 1234]{BobkovTanaka2017}. 

	On the other hand, if $p \geq 2q$ and $\gamma_u \neq 0$, then, recalling that $\varepsilon \in (0,1)$ and using  H\"older's inequality again, we can estimate the first summand in \eqref{eq:estim:Inf:G} as follows:	
	\begin{align}
	\notag
	&\intO |\nabla(\gamma_u \varphi_p+\varepsilon u^\perp)|^{q-1}|\nabla u^\perp| \, dx 
	\leq 
	\intO \left(|\gamma_u| |\nabla \varphi_p|+ |\nabla u^\perp|\right)^{q-1}|\nabla u^\perp| \, dx
	\\
	\notag
	&=
	\intO \left(|\gamma_u| |\nabla \varphi_p|+ |\nabla u^\perp|\right)^\frac{p-2}{2}|\nabla u^\perp| \cdot \left(|\gamma_u| |\nabla \varphi_p|+ |\nabla u^\perp|\right)^\frac{2q-p}{2} \, dx
	\\
	\notag
	&\leq 
	\left(\intO \left(|\gamma_u| |\nabla \varphi_p|+ |\nabla u^\perp|\right)^{p-2}|\nabla u^\perp|^2 \, dx\right)^\frac{1}{2}
	\left(\intO \frac{dx}{\left(|\gamma_u| |\nabla \varphi_p|+ |\nabla u^\perp|\right)^{p-2q}} \right)^\frac{1}{2} 
	\\
	\label{eq:g1est}
	&\leq
	C \left(|\gamma_u|^{p-2}\intO |\nabla \varphi_p|^{p-2} |\nabla u^\perp|^2 \, dx + \intO |\nabla u^\perp|^{p} \, dx \right)^\frac{1}{2}
	\frac{1}{|\gamma_u|^\frac{p-2q}{2}}\left(\intO \frac{dx}{|\nabla \varphi_p|^{p-2q}}\right)^\frac{1}{2}.
	\end{align}
	
	Let us now estimate the second term in \eqref{eq:estim:Inf:G}. 
	Since $\varepsilon \in (0,1)$ and $\|\varphi_p\|_{\infty} < +\infty$ (see Remark \ref{rem:reg}), we have 
	\begin{align*}
	\intO |\gamma_u \varphi_p+\varepsilon u^\perp|^{q-1}|u^\perp| \, dx
	\leq
	C |\gamma_u|^{q-1} \intO |u^\perp| \, dx + C \intO |u^\perp|^{q} \, dx.
	\end{align*}
	Thus, using H\"older's inequality and an embedding result of \cite[Lemma 4.2]{takac}, we get
	\begin{equation}
	\label{eq:g2est}
	\intO |\gamma_u \varphi_p + \varepsilon u^\perp|^{q-1}|u^\perp| \, dx 
	\leq
	C |\gamma_u|^{-\frac{p-2q}{2}} \left(|\gamma_u|^{p-2}\intO |\nabla \varphi_p|^{p-2} |\nabla u^\perp|^2 \, dx\right)^\frac{1}{2} + C \intO |u^\perp|^{q} \, dx.
	\end{equation}
	Finally, combining \eqref{eq:g2est} with \eqref{eq:g1est}, we obtain \eqref{eq:gest1}. 
	To obtain \eqref{eq:gest}, we apply the Sobolev embedding theorem to estimate the last summand in \eqref{eq:g2est}, and then combine \eqref{eq:g2est} with \eqref{eq:g4est}.
\end{proof}

\begin{proof}[Proof of Theorem \ref{thm:existence2}]
	We will show that, under the imposed assumptions, $E_{\lambda_1(p),\beta_*}$ attains a global minimum. 
	If $f \equiv 0$, then the existence of a global minimizer is given by \cite[Theorem 2.6 (ii)]{BobkovTanaka2017}. Thus, we may assume that $f \not\equiv 0$. 
	Let $\{u_n\} \subset \W$ be a minimizing sequence for $E_{\lambda_1(p),\beta_*}$. 
	It is not hard to see that each $E_{\lambda_1(p),\beta_*}(u_n) < 0$. Indeed, fixing an arbitrary $u \in \W$ satisfying $\langle f,u\rangle
 > 0$, and recalling that $p>q>1$, we can find a sufficiently small $t>0$ such that
	$$
	E_{\lambda_1(p),\beta_*}(t u) = \frac{t^p}{p}H_{\lambda_1(p)}(u) + \frac{t^q}{q} G_{\beta_*}(u) - t \langle f,u\rangle
 < 0.
	$$
	Let us prove that $\{u_n\}$ is bounded in $\W$.
	Suppose, by contradiction, that $\|\nabla u_n\|_p \to +\infty$ as $n \to +\infty$, up to a subsequence. 
	Making the $L^2(\Omega)$-decomposition $u_n = \gamma_n \varphi_p + u_n^\perp$ (see \eqref{eq:l2dec}), we conclude that $|\gamma_n| \to +\infty$ or $\|\nabla u_n^\perp\|_p \to +\infty$ as $n \to +\infty$. 
	To reach a contradiction, let us estimate $E_{\lambda_1(p),\beta_*}(u_n)$ from below. 
	First, to estimate $H_{\lambda_1(p)}(u_n)$, we use the improved Poincar\'e inequality obtained in \cite{takac}, which  states that 
	\begin{equation}
	\label{eq:improved-poincare}
	H_{\lambda_1(p)}(u_n) \geq 
	C |\gamma_n|^{p-2} \intO |\nabla \varphi_p|^{p-2} |\nabla u_n^\perp|^2 \, dx + C \|\nabla u_n^\perp\|_p^p.
	\end{equation}
	Hereinafter, $C>0$ is a constant independent of $n \in \mathbb{N}$.
	Second, to estimate $\langle f,u_n\rangle$, we recall that $f \in L^2(\Omega)$, $\langle f, \varphi_p\rangle = 0$, and $p > 2q > 2$, which yields
	\begin{equation}
	\label{eq:estim-f}
	\left|\langle f,u_n\rangle\right|=\left|\int_\Omega f u_n \, dx\right|
	= 
	\left|\int_\Omega f u_n^\perp \, dx\right|
	\leq \|f\|_{2} \|u_n^\perp\|_{2} \leq C \|f\|_{2} \|\nabla u_n^\perp\|_{p}.
	\end{equation}
	Finally, in order to apply Lemma \ref{lem:estimate:G} for estimating $G_{\beta_*}(u_n)$, let us show that we may assume $G_{\beta_*}(u_n) < 0$ for all $n \in \mathbb{N}$. Indeed, if $G_{\beta_*}(u_n) \geq 0$ for all $n$, up to a subsequence, then, using the estimates \eqref{eq:improved-poincare} and \eqref{eq:estim-f}, we get
	\begin{equation}\label{eq:estE0}
	E_{\lambda_1(p),\beta_*}(u_n)
	\geq 
	C |\gamma_n|^{p-2} \intO |\nabla \varphi_p|^{p-2} |\nabla u_n^\perp|^2 \, dx
	+ 
	C \|\nabla u_n^\perp\|_p^p
	- 
	C \|f\|_{2} \|\nabla u_n^\perp\|_{p}.
	\end{equation}
	If $\|\nabla u_n^\perp\|_p \to +\infty$ as $n \to +\infty$, then $E_{\lambda_1(p),\beta_*}(u_n) \to +\infty$ regardless the behavior of $\{\gamma_n\}$, which contradicts the minimization property of $\{u_n\}$. Hence, $\{u_n^\perp\}$ is bounded in $\W$. 
	Therefore, since $\|\nabla u_n\|_p \to +\infty$ as $n \to +\infty$, we have $|\gamma_n| \to +\infty$. 
	We see from \eqref{eq:estE0} that if $|\gamma_n|^{p-2}\intO |\nabla \varphi_p|^{p-2} |\nabla u_n^\perp|^2 \, dx \to +\infty$ as $n \to +\infty$, up to a subsequence, then $E_{\lambda_1(p),\beta_*}(u_n) \to +\infty$, which is again impossible. This implies $\intO |\nabla \varphi_p|^{p-2} |\nabla u_n^\perp|^2 \, dx \to 0$. 
	In view of the embedding result \cite[Lemma 4.2]{takac}, we get $\|u_n^\perp\|_2 \to 0$, and hence \eqref{eq:estim-f} yields $\langle f,u_n\rangle \to 0$. Thus, we conclude from \eqref{eq:estE0} and the behavior of $\langle f,u_n\rangle$ that $E_{\lambda_1(p),\beta_*}(u_n) \geq 0 + o(1)$ as $n \to +\infty$, which contradicts the minimization property of $\{u_n\}$ and the fact that $E_{\lambda_1(p),\beta_*}(u_n) < 0$ for each $n \in \mathbb{N}$. Thus, $G_{\beta_*}(u_n) < 0$ for all $n \in \mathbb{N}$.
	
	Substituting now \eqref{eq:improved-poincare}, \eqref{eq:estim-f}, and the estimate \eqref{eq:gest} for $G_{\beta_*}(u_n)$ (combined with H\"older's inequality) into $E_{\lambda_1(p),\beta_*}(u_n)$, we get
	\begin{align}
	\notag
	E_{\lambda_1(p),\beta_*}(u_n) 
	&\geq 
	C |\gamma_n|^{p-2} \intO |\nabla \varphi_p|^{p-2} |\nabla u_n^\perp|^2 \, dx 
	+ 
	C \|\nabla u_n^\perp\|_p^p
	\\
	\label{eq:estE}
	&-
	C |\gamma_n|^{q-1} \left(\intO |\nabla \varphi_p|^{p-2} |\nabla u_n^\perp|^2 \, dx\right)^\frac{1}{2} 
	-
	C \|\nabla u_n^\perp\|_p^q
	-
	C \|f\|_{2} \|\nabla u_n^\perp\|_{p}.
	\end{align}
	Let us consider the following three possible cases.
	
	1. $\|\nabla u_n^\perp\|_p \to +\infty$ and $|\gamma_n| \to +\infty$ as $n \to +\infty$. 
	Using Young's inequality, we have
	\begin{align}
	\notag
	|\gamma_n|^{q-1} \left(\intO |\nabla \varphi_p|^{p-2} |\nabla u_n^\perp|^2 \, dx\right)^\frac{1}{2} 
	&= 
	\varepsilon |\gamma_n|^\frac{p-2}{2} \left(\intO |\nabla \varphi_p|^{p-2} |\nabla u_n^\perp|^2 \, dx\right)^\frac{1}{2} \cdot \frac{1}{\varepsilon} |\gamma_n|^\frac{2q-p}{2}\\
	\label{eq:young}
	&\leq
	\varepsilon^2 |\gamma_n|^{p-2} \intO |\nabla \varphi_p|^{p-2} |\nabla u_n^\perp|^2 \, dx
	+
	\frac{1}{\varepsilon^2} |\gamma_n|^{2q-p}
	\end{align}
	for any $\varepsilon>0$. 
	Substituting \eqref{eq:young} into \eqref{eq:estE}, we obtain
	\begin{align}
	\notag
	E_{\lambda_1(p),\beta_*}(u_n) 
	&\geq (C - C\varepsilon^2) |\gamma_n|^{p-2} \intO |\nabla \varphi_p|^{p-2} |\nabla u_n^\perp|^2 \, dx 
	+
	C \|\nabla u_n^\perp\|_p^p
	\\
	\label{eq:estE2}
	&-
	\frac{C}{\varepsilon^2} |\gamma_n|^{2q-p}
	-
	C \|\nabla u_n^\perp\|_p^q
	-
	C \|f\|_{2} \|\nabla u_n^\perp\|_{p}.
	\end{align}
	Taking $\varepsilon>0$ small enough and recalling that $p > 2q$, we easily conclude that $E_{\lambda_1(p),\beta_*}(u_n) \to +\infty$ as $n \to +\infty$, which is impossible since $\{u_n\}$ is a minimizing sequence. 
	
	2. $\|\nabla u_n^\perp\|_p \to +\infty$ as $n \to +\infty$ and $\{\gamma_n\} $ is bounded.
	Since $\varphi_p \in C^1(\overline{\Omega})$ (see Remark \ref{rem:reg}) and $p>2$, we have
	$$
	|\gamma_n|^{q-1} \left(\intO |\nabla \varphi_p|^{p-2} |\nabla u_n^\perp|^2 \, dx\right)^\frac{1}{2} 
	\leq 
	C |\gamma_n|^{q-1} \|\nabla u_n^\perp\|_2 \leq C |\gamma_n|^{q-1} \|\nabla u_n^\perp\|_p,
	$$
	and hence, recalling that $p>q>1$, we obtain from \eqref{eq:estE} that
	$$
	E_{\lambda_1(p),\beta_*}(u_n) 
	\geq
	C \|\nabla u_n^\perp\|_p^p
	-
	C |\gamma_n|^{q-1} \|\nabla u_n^\perp\|_p
	-
	C \|\nabla u_n^\perp\|_p^q
	-
	C \|f\|_{2} \|\nabla u_n^\perp\|_{p} \to +\infty
	$$
	as $n \to +\infty$, a contradiction. 	
	
	3. $\{\|\nabla u_n^\perp\|_p\}$ is bounded and $|\gamma_n| \to +\infty$ as $n \to +\infty$.
	In this case, $\{u_n^\perp\}$ converges weakly in $\W$ and strongly in $L^r(\Omega)$, $r \in (1,p^*)$, to some $u_0^\perp \in \W$, up to a subsequence.
	Recalling that $p>2q$, we see from the estimate \eqref{eq:estE2} that if $|\gamma_n|^{p-2}\intO |\nabla \varphi_p|^{p-2} |\nabla u_n^\perp|^2 \, dx \to +\infty$ as $n \to +\infty$, up to a subsequence, then $E_{\lambda_1(p),\beta_*}(u_n) \to +\infty$, which is impossible. This implies $\intO |\nabla \varphi_p|^{p-2} |\nabla u_n^\perp|^2 \, dx \to 0$. 
	In view of the embedding result \cite[Lemma 4.2]{takac}, we get $\|u_n^\perp\|_2 \to 0$, which yields $u_0^\perp \equiv 0$ and hence $\|u_n^\perp\|_q \to 0$ due to the strong convergence of $\{u_n^\perp\}$ in $L^r(\Omega)$, $r \in (1,p^*)$.
	Consequently, we see from \eqref{eq:gest1} and \eqref{eq:estim-f} that
	$$
	|G_{\beta_*}(u_n)| \to 0 
	\quad\text{and}\quad
	\left|\langle f,u_n\rangle\right| \to 0
	\quad\text{as }
	n \to +\infty.
	$$
	Thus, 
	$$
	E_{\lambda_1(p),\beta_*}(u_n) = H_{\lambda_1(p)}(u_n) + o(1) \geq o(1)
	\quad\text{as }
	n \to +\infty,
	$$
	which contradicts the facts that each $E_{\lambda_1(p),\beta_*}(u_n) < 0$ and $\{u_n\}$ is a minimizing sequence for $E_{\lambda_1(p),\beta_*}$. 
	
	Therefore, we conclude that $\{u_n\}$ is bounded in $\W$, and hence $\inf_{\W} E_{\lambda_1(p),\beta_*}
=E_{\lambda_1(p),\beta_*}(u_n)+o(1)>-\infty$. Moreover, since $E_{\lambda_1(p),\beta_*}$ is weakly lower-semicontinuous, we see that, up to a subsequence, $\{u_n\}$ converges strongly in $\W$ to a global minimizer $u$ of $E_{\lambda_1(p),\beta_*}$, and hence $u$ is a critical point of $E_{\lambda_1(p),\beta_*}$.	
\end{proof}

\section{Proofs. Sign properties}\label{sec:proofs_sign}
In this section, we prove the results stated in Section \ref{subsec:sign}.
We start with the following auxiliary lemma which will be employed several times. 
\begin{lemma}\label{lem:weak}
	Let $\alpha > \lambda_1(p)$. Assume that $u \in \W$ satisfies $H_\alpha(u) < 0$. 
	Then there exists a sequence $\{u_n\} \subset \W$ such that $H_\alpha(u_n) = 0$ for all $n \in \mathbb{N}$, and $u_n \to u$ weakly in $\W$ and strongly in $W_0^{1,q}$ as $n \to +\infty$. 
	Moreover, if $u \geq 0$ a.e.\ in $\Omega$, then $\{u_n\}$ can be chosen such that $u_n \geq 0$ a.e.\ in $\Omega$ for all $n \in \mathbb{N}$.
\end{lemma}
\begin{proof}
	The claims can be obtained arguing in much the same way as in the proof of \cite[Theorem 2.5 (ii)]{BobkovTanaka2017} by considering $u$ instead of $\varphi_p$.
\end{proof}

Now we provide several properties of the critical value $\beta_f(\alpha)$ defined by \eqref{beta_f}. Recall that the functional $\Phi_\alpha^+$ and the set $\mathcal{B}^+(\alpha)$ are defined by \eqref{def:Phi} and \eqref{def:B^+}, respectively.

\begin{proposition}\label{lem:prop:beta_f}
	Let $f \in W^{-1,q'}(\Omega) \setminus \{0\}$ and $f \geq 0$ in the weak sense. 
	Let $\mathrm{supp}\, f$ be the support of the distribution $f$. 
	Then the following assertions are satisfied:
	\begin{enumerate}[label={\rm(\roman*)}]
		\item\label{lem:prop:beta_f:1} $\Phi_\alpha^+$ is 0-homogeneous and $\mathcal{B}^+(\alpha) \neq \emptyset$ for all $\alpha \in \mathbb{R}$; 
		\item\label{lem:prop:beta_f:2} if $\alpha < \alpha_*$, then $\beta_f(\alpha) \geq \lambda_1(q)$, and $\beta_f(\alpha) > \lambda_1(q)$ if and only if $\langle f, \varphi_q \rangle > 0$;
		\item\label{lem:prop:beta_f:3} if $\alpha \geq \alpha_*$, then $\beta_f(\alpha) = \lambda_1(q)$;
		\item\label{lem:prop:beta_f:4} if $\alpha \geq \lambda_1(p)$, then $\beta_f(\alpha) \leq \beta_*$;
		\item\label{lem:prop:beta_f:5} $\beta_f$ is nonincreasing and left-continuous;
		\item\label{lem:prop:beta_f:9} if there is a nonempty open set $\widetilde{\Omega} \subset \Omega \setminus \mathrm{supp}\, f$, then 
		$\sup \{\beta_f(\alpha): \alpha\le \lambda_1(p)\} < +\infty$.
	\end{enumerate}
	If, in addition, $\langle f, v \rangle >0$ for any $v \in W_0^{1,q} \setminus \{0\}$ with $v\ge 0$ a.e.\ in $\Omega$, 
then the following assertions are satisfied:
	\begin{enumerate}[label={\rm(\roman*)}]\setcounter{enumi}{6}
		\item\label{lem:prop:beta_f:7} $\beta_f$ is continuous in $(-\infty,\lambda_1(p))$; 
		\item\label{lem:prop:beta_f:8} $\beta_f(\alpha) \to +\infty$ as $\alpha\to-\infty$. 
	\end{enumerate}
\end{proposition}
\begin{proof}
	Assertion \ref{lem:prop:beta_f:1}. The $0$-homogeneity of $\Phi_\alpha^+$ is trivial. 
	To obtain $\mathcal{B}^+(\alpha) \neq \emptyset$ we consider an eigenfunction $u$ of the $p$-Laplacian associated with some eigenvalue $\lambda_k(p) > \max\{\alpha,\lambda_1(p)\}$. Then $u^+ \not\equiv 0$ and $H_\alpha(u^+)=(\lambda_k(p)-\alpha)\|u^+\|_p^p>0$, which yields $u^+\in\mathcal{B}^+(\alpha)$. 
	
	Assertion \ref{lem:prop:beta_f:2}. The inequality $\beta_f(\alpha) \geq \lambda_1(q)$ is trivial. Moreover, if $\langle f, \varphi_q \rangle = 0$, then $\beta_f(\alpha) = \lambda_1(q)$ by the definition \eqref{beta_f} of $\beta_f(\alpha)$ since $\varphi_q \in \mathcal{B}^+(\alpha)$. Let $\langle f, \varphi_q \rangle > 0$ and suppose, by contradiction, that there is some $\alpha < \alpha_*$ such that $\beta_f(\alpha) = \lambda_1(q)$. Then, in view of assertion \ref{lem:prop:beta_f:1}, we can find a minimizing sequence 
	$\{u_n\}\subset \mathcal{B}^+(\alpha)$ for $\beta_f(\alpha)$ such that 
	$\|u_n\|_q = 1$ for all $n \in \mathbb{N}$, 
	\begin{equation}\label{eq:lem:beta_f:1}
	\|\nabla u_n\|_q^q \to \lambda_1(q) 
	\quad 
	\text{and}
	\quad 
	\left(H_\alpha(u_n)\right)^\frac{q-1}{p-1} \langle f,u_n\rangle^\frac{p-q}{p-1} \to 0
	\end{equation}
	as $n \to +\infty$. 
	Passing to a subsequence, we see that $u_n \to \varphi_q$ strongly in $W_0^{1,q}$ and $L^q(\Omega)$, where we assumed $\|\varphi_q\|_q=1$. 
	Therefore,
	$$
	\lim\limits_{n\to+\infty}\langle f,u_n\rangle = \langle f,\varphi_q\rangle > 0,
	$$
	and hence the second convergence in \eqref{eq:lem:beta_f:1} implies 
	$\lim\limits_{n\to+\infty} H_\alpha(u_n)=0$. 
	Let us show that $\{u_n\}$ is bounded in $\W$. Indeed, if we suppose, by contradiction, that $\|\nabla u_n\|_p \to +\infty$ as $n \to +\infty$, up to a subsequence, then $\lim\limits_{n\to+\infty} H_\alpha(u_n)=0$ implies $\|\nabla u_n\|_p^p\le (\alpha+1)\|u_n\|_p^p$ for all sufficiently large $n \in \mathbb{N}$. 
	Thus, according to \cite[Lemma 9]{T-2014}, there exists a constant 
	$C=C(\alpha)>0$ independent of $n$ such that 
	\begin{equation}
	\label{eq:un<C}
	\|\nabla u_n\|_p \leq C\|u_n\|_q=C,
	\end{equation}
	which is a contradiction, and hence $\{u_n\}$ is bounded in $\W$.
	Recalling that $u_n \to \varphi_q$ strongly in $W_0^{1,q}$, we conclude from \eqref{eq:un<C} that $u_n\to \varphi_q$ weakly in $W_0^{1,p}$, up to a subsequence, whence we get 
	\begin{equation}
	\label{eq:H<=0}
	H_\alpha(\varphi_q) \leq \liminf_{n \to +\infty} H_\alpha(u_n) = 0.
	\end{equation}
	However, by the definition \eqref{eq:alpha_*} of $\alpha_*$, \eqref{eq:H<=0} contradicts the assumption $\alpha < \alpha_*$.
	
	Assertion \ref{lem:prop:beta_f:3}. 
	Note that $\beta_f(\alpha_*) = \lambda_1(q)$ since $\varphi_q \in \mathcal{B}^+(\alpha_*)$. 
	Thus, let us assume $\alpha >\alpha_*$. Then $H_\alpha(\varphi_q) < 0$, and hence applying Lemma \ref{lem:weak} to $\varphi_q$, we can find a sequence $\{u_n\} \subset \W \setminus \{0\}$ such that $u_n \geq 0$ a.e.\ in $\Omega$ and $H_\alpha(u_n) = 0$ for any $n \in \mathbb{N}$, and 
	$\frac{\|\nabla u_n\|_q^q}{\|u_n\|_q^q} \to \lambda_1(q)$ as $n \to +\infty$. That is, $\{u_n\} \subset \mathcal{B}^+(\alpha)$, and the assertion follows.

	Before proving assertion \ref{lem:prop:beta_f:4}, let us establish assertion \ref{lem:prop:beta_f:5}. 
	We start with the monotonicity of $\beta_f$. Suppose, by contradiction, that there exist $\alpha_1$ and $\alpha_2$ such that $\alpha_1 < \alpha_2$ and $\beta_f(\alpha_1) < \beta_f(\alpha_2)$. That is, we can find $u \in \mathcal{B}^+(\alpha_1)$ such that $\Phi^+_{\alpha_1}(u) < \beta_f(\alpha_2)$. If $H_{\alpha_2}(u) \geq 0$, then $u \in \mathcal{B}^+(\alpha_2)$. Moreover, $\alpha_1 < \alpha_2$ implies $H_{\alpha_1}(u) > H_{\alpha_2}(u)$ and hence $\Phi_{\alpha_1}^+(u) \geq \Phi_{\alpha_2}^+(u)$, which contradicts the definition of $\beta_f(\alpha_2)$. Therefore, $H_{\alpha_2}(u) < 0$. Applying Lemma \ref{lem:weak}, we can find a sequence $\{u_n\} \subset \W$ such that $u_n \geq 0$ a.e.\ in $\Omega$, 
	\begin{equation}
	\label{eq:hphi}
	H_{\alpha_2}(u_n) = 0
	\quad \text{and} \quad
	\Phi^+_{\alpha_2}(u_n) = \frac{\|\nabla u_n\|_q^q}{\|u_n\|_q^q} \to \frac{\|\nabla u\|_q^q}{\|u\|_q^q} \leq \Phi^+_{\alpha_1}(u)
	\quad \text{as }
	n \to +\infty.
	\end{equation}
	We see that $\{u_n\} \subset \mathcal{B}^+(\alpha_2)$, and hence \eqref{eq:hphi} leads to a contradiction since $\Phi^+_{\alpha_1}(u) < \beta_f(\alpha_2)$.
	
	Now we prove the left-continuity of $\beta_f$.
	Let us fix some $\alpha \in \mathbb{R}$ and consider a sequence $\{\alpha_n\}$ such that 
	$\alpha_n < \alpha$ for all $n \in \mathbb{N}$ and $\lim\limits_{n\to+\infty}\alpha_n=\alpha$. 
	Since we know that $\liminf\limits_{n\to+\infty}\beta_f(\alpha_n) \ge \beta_f(\alpha)$ by the monotonicity obtained above, let us show that $\limsup\limits_{n\to+\infty}\beta_f(\alpha_n)\le \beta_f(\alpha)$.
	Arguing by contradiction, we suppose that, up to a subsequence, 
	$$
	\delta:=\lim_{n\to+\infty} \beta_f(\alpha_n)>\beta_f(\alpha). 
	$$
	Thus, we can find $u \in \mathcal{B}^+(\alpha)$	satisfying $\Phi_{\alpha}^+(u)<\delta$. 
	Since each $\alpha_n < \alpha$, we have $u \in \mathcal{B}^+(\alpha_n)$. Therefore, recalling that $\alpha_n \to \alpha$ as $n \to +\infty$, we get a contradiction by 
	$$
	\delta=\lim_{n\to+\infty}\beta_f(\alpha_n)\le \lim_{n\to+\infty}\Phi_{\alpha_n}^+(u)=\Phi_{\alpha}^+(u)<\delta. 
	$$	
		
	Assertion \ref{lem:prop:beta_f:4} easily follows from assertion \ref{lem:prop:beta_f:5} by noting that $\beta_f(\lambda_1(p)) \le \beta_*$ in view of $\varphi_p \in \mathcal{B}^+(\lambda_1(p))$.
	
	Assertion \ref{lem:prop:beta_f:9}.	
	Taking a nonnegative function $u \in C_0^\infty(\Omega) \setminus \{0\}$ such that $\text{supp}\, u \subset \widetilde{\Omega}$, we see that 
	$H_\alpha(u) \geq 0$ for all $\alpha \leq \lambda_1(p)$ and 
	$\langle f,u\rangle=0$, which yields the desired bound: 
	$$
	\beta_f(\alpha) \leq \Phi_\alpha^+(u)=\frac{\|\nabla u\|_q^q}{\|u\|_q^q}
	\quad \text{for any } 
	\alpha \leq \lambda_1(p).
	$$

	Assertion \ref{lem:prop:beta_f:7}. Let $\alpha < \lambda_1(p)$ and let
	$\{\alpha_n\}$ be an arbitrary sequence convergent to $\alpha$. 
	Since we already know that $\beta_f$ is nonincreasing and left-continuous  
	by assertion \ref{lem:prop:beta_f:5}, 
	it is sufficient to assume that $\alpha < \alpha_n < \lambda_1(p)$ and to show that 
	$\liminf\limits_{n\to+\infty}\beta_f(\alpha_n) \ge \beta_f(\alpha)$. 
	Suppose, by contradiction, that, up to a subsequence, $\lim\limits_{n\to+\infty}\beta_f(\alpha_n) < \beta_f(\alpha)$. Thus, for any sufficiently large $n \in \mathbb{N}$ we can choose $u_n \in \mathcal{B}^+(\alpha_n)$ such that 
	\begin{equation}
	\label{eq:b<b<b}
	\beta_f(\alpha_n) \leq \Phi_{\alpha_n}^+(u_n) 
\le \sup_{m}\Phi_{\alpha_m}^+(u_m) < \beta_f(\alpha), 
	\end{equation}
	and we may assume $\|u_n\|_q=1$ for all $n$. The latter inequality in \eqref{eq:b<b<b} implies the existence of $C>0$ such that
	\begin{equation}
	\label{eq:hhf<c}
	\|\nabla u_n\|_q^q \leq C
	\quad \text{and} \quad 
	\left(H_{\alpha_n}(u_n)\right)^\frac{q-1}{p-1}\langle f,u_n\rangle^\frac{p-q}{p-1} \leq C 
	\end{equation}
	for all $n$. 
	In view of the first bound in \eqref{eq:hhf<c} and the choice $\|u_n\|_q=1$, we see that $\{u_n\}$ converges to some nonnegative function $u_0 \in W_0^{1,q} \setminus \{0\}$ weakly in $W_0^{1,q}$ and strongly in $L^q(\Omega)$, up to a subsequence. 
	By our assumption, we get 
	$\lim\limits_{n\to+\infty}\langle f,u_n\rangle= \langle f,u_0\rangle>0$, which implies the uniform boundedness of $H_{\alpha_n}(u_n)$ for all $n$. 
	Hence, by the same argument as in assertion \ref{lem:prop:beta_f:2} (see \eqref{eq:un<C}), we conclude that $\{u_n\}$ is bounded in $\W$, and thereby it converges to $u_0$ weakly in $\W$, up to a subsequence. Moreover, $H_\alpha(u_0) \geq 0$ by $\alpha\le \lambda_1(p)$.
	Thus, we see that $u_0 \in \mathcal{B}^+(\alpha)$ and we deduce from \eqref{eq:b<b<b} that
	$$
	\Phi_{\alpha}^+(u_0) \leq \liminf_{n\to+\infty}\Phi_{\alpha_n}^+(u_n) < 
	\beta_f(\alpha),
	$$
	which contradicts the definition of $\beta_f(\alpha)$. 
	
	Assertion \ref{lem:prop:beta_f:8}. Suppose, by contradiction, that 
	$C := \sup\{\beta_f(\alpha): \alpha\le \lambda_1(p)\} < +\infty$. 
	The monotonicity of $\beta_f$ (see assertion \ref{lem:prop:beta_f:5}) implies the existence of a sequence $\{\alpha_n\}$ such that $\alpha_n \to -\infty$ and  $\beta_f(\alpha_n) \to C$ as $n \to +\infty$. 
	Therefore, by the definition of $\beta_f(\alpha_n)$, we can find a sequence $\{u_n\} \subset \mathcal{B}^+(\alpha_n)$ such that $\|u_n\|_q = 1$ for each $n \in \mathbb{N}$ and 
	$\Phi_{\alpha_n}^+(u_n)  \leq C+1$. 
	Arguing as in assertion \ref{lem:prop:beta_f:7} (see \eqref{eq:hhf<c}), we obtain that $u_n \to u_0 \not \equiv 0$ weakly in $W_0^{1,q}$ and strongly in $L^q(\Omega)$, up to a subsequence. Moreover, $H_{\alpha_n}(u_n)$ is uniformly bounded for all $n$. Since $\alpha_n \to -\infty$ as $n \to +\infty$, we conclude that $\{u_n\}$ is bounded in $\W$ and $\|u_n\|_p \to 0$. However, this contradicts $u_0 \not\equiv 0$.  
\end{proof}

The following result is crucial for the proof of Theorem \ref{thm:nodal} and Proposition \ref{prop:aux5} \ref{prop:aux5:2}.
\begin{proposition}\label{prop:aux}
	Let $f \in W^{-1,p'}(\Omega)$ and $f \geq 0$ in the weak sense. 
	Assume that $\alpha \in \mathbb{R}$ and $\beta < \beta_f(\alpha)$.
	If $u \in \W$ is such that $u \geq 0$ a.e.\ in $\Omega$ and $H_\alpha(u) > 0$, then 
	\begin{equation}\label{eq:aux}
	H_\alpha(u) + G_\beta(u) + \langle f,u\rangle > 0.
	\end{equation}
\end{proposition}
\begin{proof}
	If $u \in \W$ is such that $u \geq 0$ a.e.\ in $\Omega$ and $H_\alpha(u) > 0$, then $\langle f,u\rangle\geq 0$. 
	Evidently, if $G_\beta(u) \geq 0$, then \eqref{eq:aux} is satisfied. Thus, let us assume that $G_\beta(u) < 0$. Note that $u \in \mathcal{B}^+(\alpha)$, i.e., $u$ is admissible for the minimization problem \eqref{beta_f} of $\beta_f(\alpha)$. Consequently, if $\langle f,u\rangle = 0$, then $\beta < \beta_f(\alpha)$ implies $G_\beta(u) > 0$, which is impossible. Therefore, 
$\langle f,u\rangle > 0$.
	
	Consider the function
	$$
	Q_\beta(t) := t^p H_\alpha(u) + t^q G_\beta(u) + t \langle f,u\rangle, 
	\quad
	t \geq 0.
	$$
	Since $p>q>1$, we see that $t^p H_\alpha(u) > 0$ is the leading term at $t \to +\infty$,  $t \langle f,u\rangle > 0$ is the leading term at $t \to 0$, and $t^q G_\beta(u) < 0$ has an impact in a middle range of $t$, see Figure \ref{fig1}.
	
	If $Q_\beta(1) > 0$, then \eqref{eq:aux} is satisfied. Thus, let us suppose that $Q_\beta(1) \leq 0$. 
	In view of the behavior of $Q_\beta(t)$, we can find $t_0, t_1>0$ such that $t_0<1<t_1$, $Q_\beta(t_0) > 0$, and $Q_\beta(t_1) > 0$. 
	Let us define the value
	$$
	q(b) := \min_{t_0 \leq t \leq t_1} Q_b(t),
	$$
	and let $t_b \in (t_0, t_1)$ be a corresponding minimizer. We see that $q(\beta) = Q_{\beta}(t_{\beta}) \leq 0$ and $Q_{\beta}'(t_{\beta}) = 0$. 
	Note that $Q_b(t)$ is strictly decreasing with respect to $b$ for any fixed $t>0$. Moreover, if $\hat{\beta} < \beta$ is such that $G_{\hat{\beta}}(u) \geq 0$, then $Q_{\hat{\beta}}(t) \geq q(\hat{\beta}) \geq c > 0$ for some constant $c$ and all $t \geq t_0$. Therefore, noting that $q(\cdot)$ is continuous, we obtain the existence of $\tilde{\beta} \in (\hat{\beta}, \beta]$ such that
	\begin{equation}\label{eq:qq=0}
	Q_{\tilde{\beta}}(t_{\tilde{\beta}}) = 0 \quad \text{and} \quad Q_{\tilde{\beta}}'(t_{\tilde{\beta}}) = 0.
	\end{equation}

	\begin{figure}[!h]
		\begin{center}
			\includegraphics[width=0.5\linewidth]{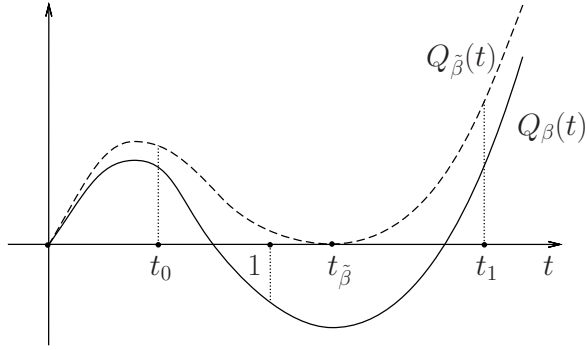}
			\caption{Schematic behavior of $Q_\beta(t)$ and $Q_{\tilde{\beta}}(t)$.}
			\label{fig1}
		\end{center}
	\end{figure}

	Let us denote for simplicity $t = t_{\tilde{\beta}}$, and rewrite \eqref{eq:qq=0} as 
	$$
	\left\{
	\begin{aligned}
	&t^p H_\alpha(u) + t^q G_{\tilde{\beta}}(u) + t \langle f,u\rangle = 0,\\
	&p t^p H_\alpha(u) + qt^q G_{\tilde{\beta}}(u) + t \langle f,u\rangle = 0.
	\end{aligned}
	\right.
	$$
	Solving this system with respect to $G_{\tilde{\beta}}(u)$ and $t$, we obtain
	$$
	t = \left(\frac{q-1}{p-q}\right)^\frac{1}{p-1}\left(\frac{\langle f,u\rangle}{H_\alpha(u)}\right)^\frac{1}{p-1}
	$$
	and
	\begin{equation}\label{eq:beta_f:contr2}
	G_{\tilde{\beta}}(u) =  -\frac{p-1}{p-q}\left(\frac{p-q}{q-1}\right)^\frac{q-1}{p-1}\left(H_\alpha(u)\right)^\frac{q-1}{p-1}\langle f,u\rangle^\frac{p-q}{p-1}.
	\end{equation}
	Expressing now $\tilde{\beta}$ from $G_{\tilde{\beta}}(u)$, we get
	\begin{equation}\label{eq:beta_f:contr1}
	\tilde{\beta} =
	\frac{\|\nabla u\|_q^q}{\|u\|_q^q} + \frac{p-1}{p-q}\left(\frac{p-q}{q-1}\right)^\frac{q-1}{p-1}\frac{\left(H_\alpha(u)\right)^\frac{q-1}{p-1}\langle f,u\rangle^\frac{p-q}{p-1}}{\|u\|_q^q}.
	\end{equation}
	However, this is impossible since $\tilde{\beta} \leq \beta < \beta_f(\alpha)$, and $u$ is an admissible function for the minimization formulation of $\beta_f(\alpha)$. 
\end{proof}

\medskip
\begin{proof}[Proof of Theorem \ref{thm:nodal}]
	Suppose that there is a solution $u$ of \eqref{eq:D} with $\alpha \leq \lambda_1(p)$ and $\beta < \beta_f(\alpha)$
	such that $u = u^+ + u^-$ and $u^- \not\equiv 0$. 
	Then, $-u^- \geq 0$ and we have
	\begin{equation}\label{eq:HG=0}
	\left<E_{\alpha,\beta}'(u),u^-\right> =  H_\alpha(-u^-) + G_\beta(-u^-) + \langle f,-u^-\rangle = 0.
	\end{equation}
	Thus, if $\alpha<\lambda_1(p)$, then $H_\alpha(-u^-) > 0$, and hence we get a contradiction to Proposition \ref{prop:aux}. 
	In the case $\alpha=\lambda_1(p)$ it is enough to show that $u^- \not \in \mathbb{R}\varphi_p$.
	Suppose, by contradiction and without loss of generality, that $-u^- = \varphi_p$. 
	Then we get
	$$
	\left<E_{\alpha,\beta}'(u),u^-\right> = 
	G_\beta(\varphi_p) + \langle f, \varphi_p \rangle = 0,
	$$
	and hence $G_\beta(\varphi_p) \le 0$.
	On the other hand, since $\beta<\beta_f(\lambda_1(p)) \leq \beta_*$ by Lemma \ref{lem:prop:beta_f}  \ref{lem:prop:beta_f:4}, we have $G_\beta(\varphi_p) > 0$, which leads to a contradiction.
	Therefore, any solution of \eqref{eq:D} is nonnegative. 
	Moreover, if $f\in L^\infty(\Omega) \setminus \{0\}$ and $f \geq 0$ a.e.\ in $\Omega$, then we have $u\in \text{int}\, C_0^1(\overline{\Omega})_+$ (see Remark \ref{ref:sign}), that is, $u$ is positive.
\end{proof}

\medskip
\begin{proof}[Proof of Proposition \ref{prop:nodal2}]
	Let us fix some $\beta < \beta_f(\lambda_1(p))$ and consider an arbitrary sequence $\{\alpha_n\}$ such that each $\alpha_n > \lambda_1(p)$ and $\alpha_n \to \lambda_1(p)$ as $n \to +\infty$. Let $u_n$ be a solution of \eqref{eq:D} with $\alpha=\alpha_n$. 
	Assume first that $\{u_n\}$ is bounded in $\W$. Then, we deduce from Remark \ref{rem:reg} and the Arzel\`a-Ascoli theorem that $u_n \to u$ in $C^1(\overline{\Omega})$, up to a subsequence, where $u$ is a solution of \eqref{eq:D}. Noting that $u \in \text{int}\, C_0^1(\overline{\Omega})_+$ by Theorem \ref{thm:nodal} and Remark \ref{ref:sign}, we conclude that $u_n \in \text{int}\, C_0^1(\overline{\Omega})_+$ for all sufficiently large $n \in \mathbb{N}$. 
	
	Assume now that $\{u_n\}$ is unbounded in $\W$. Considering the normalized sequence $v_n := u_n/\|\nabla u_n\|_p$, we can argue as in \cite[Lemma 3.3]{BobkovTanaka2016} to obtain that $\{v_n\}$ converges strongly in $\W$ to $\varphi_p$ or $-\varphi_p$, up to a subsequence. Applying again Remark \ref{rem:reg} and the Arzel\`a-Ascoli theorem, we deduce that either $v_n \to \varphi_p$ or $v_n \to -\varphi_p$ in $C^1(\overline{\Omega})$, up to a subsequence. Since $\varphi_p \in \text{int}\, C_0^1(\overline{\Omega})_+$ (see Remark \ref{ref:sign}), we get either $v_n \in \text{int}\, C_0^1(\overline{\Omega})_+$ or $-v_n \in \text{int}\, C_0^1(\overline{\Omega})_+$ for all sufficiently large $n \in \mathbb{N}$.
	
	Finally, recalling that the sequence $\{\alpha_n\}$ has been chosen in an arbitrary way, we finish the proof.
\end{proof}

\medskip
\begin{proof}[Proof of Proposition \ref{thm:ATM}]
	Let $\alpha \geq \lambda_1(p)$ and $\beta > \beta_{ps}(\alpha)$.
	Suppose, by contradiction, that \eqref{eq:D} has a nonnegative solution $u$ for some $f\in L^\infty(\Omega)$ satisfying $f\ge0$ and $f\not\equiv0$. 
	Then $u\in \text{int}\, C_0^1(\overline{\Omega})_+$ (see Remark \ref{ref:sign}), and hence $u$ is a super-solution of \eqref{eq:D0}.
	Noting that $\beta>\beta_{ps}(\alpha)$ implies $\beta>\lambda_1(q)$ (see \cite[Proposition 3 (ii)]{BobkovTanaka2015}) and applying the sub-supersolution method \protect{\cite[Lemma 6]{BobkovTanaka2015}}, we obtain a positive solution of \eqref{eq:D0}. However, it contradicts \cite[Theorem 2.2]{BobkovTanaka2015} which asserts that \eqref{eq:D0} has no positive solutions.
\end{proof}

\medskip
\begin{proof}[Proof of Proposition \ref{thm:ATM2}]
Suppose, by contradiction, that there exists $\alpha > \alpha_*$ such that for any $n \in \mathbb{N}$ there exists $\beta_n > \lambda_1(q) - \frac{1}{n}$ and a nonnegative solution $u_n$ of \eqref{eq:D} with $\beta=\beta_n$.
Since $u_n \in \text{int}\, C_0^1(\overline{\Omega})_+$ for each $n \in \mathbb{N}$ (see Remark \ref{ref:sign}), 
we can choose $\varphi_q^p/u_n^{p-1}$ as a test function for \eqref{eq:D}. 
Then, we obtain 
\begin{align*}
\int_\Omega |\nabla u_n|^{p-2}  \nabla u_n \nabla \left( \frac{\varphi_q^p}{u_n^{p-1}} \right)\, dx
&+
\int_\Omega |\nabla u_n|^{q-2}  \nabla u_n \nabla \left( \frac{\varphi_q^p}{u_n^{p-1}} \right)\, dx
\\ 
&=
\alpha \int_\Omega \varphi_q^p \, dx 
+
\beta_n \int_\Omega \varphi_q^p u_n^{q-p} \, dx
+\intO f \frac{\varphi_q^p}{u_n^{p-1}}\,dx.
\end{align*}
The classical Picone identity (see \cite[Theorem 1.1]{Alleg}) implies
\begin{equation*}
\label{eq:equiv:pic1}
\int_\Omega |\nabla u_n|^{p-2}  \nabla u_n \nabla \left( \frac{\varphi_q^p}{u_n^{p-1}} \right)\, dx \leq 
\int_\Omega |\nabla \varphi_q|^p \, dx = 
\alpha_* \int_\Omega \varphi_q^p \, dx
\end{equation*}
and the generalized Picone's identity from \cite[Lemma~1]{ilyas} yields
\begin{align*}
\label{eq:equiv:pic2}
\int_\Omega |\nabla u_n|^{q-2}  \nabla u_n \nabla \left( \frac{\varphi_q^p}{u_n^{p-1}} \right)\, dx &\leq
\int_\Omega |\nabla \varphi_q|^{q-2} \nabla \varphi_q \nabla \left( \frac{\varphi_q^{p-q+1}}{u_n^{p-q}} \right) \, dx 
= \lambda_1(q) \int_\Omega \varphi_q^{p} u_n^{q-p} \, dx.
\end{align*} 
Consequently, we get 
\begin{equation}\label{eq:ATM-2-1} 
(\alpha-\alpha_*)\|\varphi_q\|_p^p
+(\beta_n-\lambda_1(q))\,\int_\Omega \varphi_q^p u_n^{q-p} \, dx +\intO f \frac{\varphi_q^p}{u_n^{p-1}}\,dx \leq 0
\end{equation} 
for each $n \in \mathbb{N}$. 
Note that the first and third terms in \eqref{eq:ATM-2-1} are positive, and the first term does not depend on $n$. 
Thus, if $\beta_n \geq \lambda_1(q)$ for some $n \in \mathbb{N}$, we get a contradiction. Therefore, we may assume that $\beta_n \uparrow \lambda_1(q)$, and hence \eqref{eq:ATM-2-1} implies 
\begin{equation}
\label{eq:inttoinfty}
\int_\Omega \varphi_q^p u_n^{q-p} \, dx \to +\infty
\quad\text{as }
n \to +\infty.
\end{equation}
Let us show that this is impossible. 
Since $\alpha_* > \lambda_1(p)$ and each $u_n > 0$, we can argue as in \cite[Lemma 3.3]{BobkovTanaka2016} to prove that $\{\|\nabla u_n\|_p\}$ is bounded. Hence, by Remark \ref{rem:reg} and the Arzel\`a-Ascoli theorem we obtain that $u_n \to u$ in $C^1(\overline{\Omega})$, up to a subsequence, where $u \in \text{int}\, C_0^1(\overline{\Omega})_+$ is a solution of \eqref{eq:D}.
Since $\varphi_q, u$, and each $u_n$ satisfy the boundary point lemma (see, e.g., \cite[Theorem 5.5.1]{puser}), the $C^1(\overline{\Omega})$-convergence implies that we can find $c_1, c_2 > 0$ such that 
$$
c_1\, \text{dist}(x,\partial \Omega) < \varphi_q(x), u(x), u_n(x) < c_2\, \text{dist}(x,\partial \Omega)
$$
for any $x \in \Omega$ and all sufficiently large $n \in \mathbb{N}$.
Thus, we get
$$
\int_{\Omega} \varphi_q^p u_n^{q-p} \, dx \leq 
c_2^p c_1^{q-p} \int_{\Omega} (\text{dist}(x,\partial \Omega))^q \, dx < +\infty,
$$
which contradicts \eqref{eq:inttoinfty}.
\end{proof}

Let us turn to the proof of Proposition \ref{prop:aux5}.
We start by showing some basic properties of the critical value $\beta^f(\alpha)$ defined by \eqref{beta^f}. Recall that the functional $\Phi_\alpha^-$ and the set $\mathcal{B}^-(\alpha)$ are defined by \eqref{def:Phi^-} and \eqref{def:B^-}, respectively.
\begin{lemma}\label{lem:prop:beta^f}
	Let $f \in W^{-1,p'}(\Omega) \setminus \{0\}$ and $f \geq 0$ in the weak sense. 
	Then the following assertions are satisfied:
	\begin{enumerate}[label={\rm(\roman*)}]
		\item\label{lem:prop:beta^f:1} 
		$\Phi_\alpha^-$ is 0-homogeneous and $\mathcal{B}^-(\alpha)\not=\emptyset$ for any $\alpha\ge \lambda_1(p)$;
		\item\label{lem:prop:beta^f:2} $\beta^f(\alpha) < +\infty$ for any $\alpha\ge \lambda_1(p)$;
		\item\label{lem:prop:beta^f:3} $\beta^f$ is nondecreasing in $[\lambda_1(p), +\infty)$; 
		\item\label{lem:prop:beta^f:4} $\beta^f(\lambda_1(p)) = \beta_*$ and $\beta^f(\alpha) \to +\infty$ as $\alpha \to +\infty$.
	\end{enumerate}
\end{lemma}
\begin{proof}
	Assertion \ref{lem:prop:beta^f:1} is trivial.
	Let us prove assertion \ref{lem:prop:beta^f:2}. Note that 
	$$
	\mathcal{B}^-(\alpha) \subset 
	X(\alpha) 
	:= 
	\left\{u \in \W:~ \|\nabla u\|_p^p \leq \alpha \|u\|_p^p\right\}
	$$
	for all $\alpha \geq \lambda_1(p)$, and hence \cite[Lemma 9]{T-2014} implies the existence of a constant $C=C(\alpha)>0$ such that 
	$$
	\|\nabla u\|_p \leq C\|u\|_q 
	\quad \text{for any } u\in \mathcal{B}^-(\alpha).
	$$
	Applying H\"older's inequality, we get $\|\nabla u\|_q \leq |\Omega|^\frac{p-q}{pq} \|\nabla u\|_p$, which yields the desired boundedness:
	$$
	\beta^f(\alpha) \leq \Phi_\alpha^-(u) \leq \frac{\|\nabla u\|_q^q}{\|u\|_q^q} \le |\Omega|^\frac{p-q}{p} C^q < +\infty
	\quad \text{for any } u\in \mathcal{B}^-(\alpha).
	$$
	
	Assertion \ref{lem:prop:beta^f:3}. 
	We argue similarly to the proof of Proposition \ref{lem:prop:beta_f} \ref{lem:prop:beta_f:5}.
	Suppose, by contradiction, that there exist $\alpha_1, \alpha_2 \geq \lambda_1(p)$ such that $\alpha_1 < \alpha_2$ and $\beta^f(\alpha_1) > \beta^f(\alpha_2)$. That is, we can find $u \in \mathcal{B}^-(\alpha_1)$ such that $\Phi^-_{\alpha_1}(u) > \beta^f(\alpha_2)$. Since $H_{\alpha_1}(u) \leq 0$, we have $H_{\alpha_2}(u) < H_{\alpha_1}(u) \leq 0$. Applying Lemma \ref{lem:weak}, we can find a sequence $\{u_n\} \subset \W$ such that $u_n \geq 0$ a.e.\ in $\Omega$,
	\begin{equation}
	\label{eq:hphi2}
	H_{\alpha_2}(u_n) = 0
	\quad \text{and} \quad
	\Phi^-_{\alpha_2}(u_n) = \frac{\|\nabla u_n\|_q^q}{\|u_n\|_q^q} \to \frac{\|\nabla u\|_q^q}{\|u\|_q^q} \geq \Phi^-_{\alpha_1}(u)
	\quad \text{as }
	n \to +\infty.
	\end{equation}
	We see that $\{u_n\} \subset \mathcal{B}^+(\alpha_2)$, and hence \eqref{eq:hphi2} leads to a contradiction since $\Phi^-_{\alpha_1}(u) > \beta^f(\alpha_2)$. 
	
	Assertion \ref{lem:prop:beta^f:4}. 
	The equality $\beta^f(\lambda_1(p)) = \beta_*$ is trivial. 
	To show that $\beta^f(\alpha) \to +\infty$ as $\alpha \to +\infty$, let us assume, without loss of generality, that $0 \in \Omega$, and let us fix a ball $B \subset \Omega$ such that $0 \in B$. Consider any nonnegative $u \in C_0^\infty(\Omega)$ and $\alpha \geq \lambda_1(p)$ such that $\text{supp}\, u \subset B$ and $\|\nabla u\|_p^p = \alpha \|u\|_p^p$, i.e., $H_\alpha(u)=0$. Now we define a function $u_n$ by $u_n(x) = u(nx)$ for each $n \in \mathbb{N}$. Since $\text{supp}\, u \subset B$, we get $u_n \in C_0^\infty(\Omega)$. Moreover, it is not hard to obtain that 
	\begin{equation}
	\label{eq:unhn}
	\|\nabla u_n\|_p^p = \alpha n^p \|u_n\|_p^p
	\quad\text{and}\quad
	\Phi^-_{\alpha n^p}(u_n) = \frac{\|\nabla u_n\|_q^q}{\|u_n\|_q^q} = n^q \frac{\|\nabla u\|_q^q}{\|u\|_q^q}.
	\end{equation}
	Thus, we see from \eqref{eq:unhn} that $u_n \in \mathcal{B}^-(\alpha n^p)$ and $\beta^f(\alpha n^p) \geq \Phi^-_{\alpha n^p}(u_n) \to +\infty$ as $n \to +\infty$.
\end{proof}

The following result can be obtained in much the same way as Proposition \ref{prop:aux}. 
\begin{proposition}\label{prop:aux4}
	Let $f \in W^{-1,p'}(\Omega)$ and $f \geq 0$ in the weak sense. 
	Assume that $\alpha > \lambda_1(p)$ and $\beta > \beta^f(\alpha)$.
	If $u \in \W$ is such that $u \geq 0$ a.e.\ in $\Omega$ and $H_\alpha(u) < 0$, then 
	\begin{equation}\label{eq:aux4}
	H_\alpha(u) + G_\beta(u) - \langle f,u\rangle< 0.
	\end{equation}
\end{proposition}

\begin{proof}[Proof of Proposition \ref{prop:aux5}]
	\ref{prop:aux5:1} Let $\beta < \beta_f(\alpha)$ and $u^- \not\equiv 0$. Suppose, by contradiction, that $H_\alpha(u^-) \geq 0$. As in the proof of Theorem \ref{thm:nodal}, $-u^-$ satisfies the equality \eqref{eq:HG=0}. Therefore, if $H_\alpha(u^-) > 0$, then we get a contradiction to Proposition \ref{prop:aux} applied to $-u^-$. If $H_\alpha(u^-) = 0$, then $-u^- \in \mathcal{B}^+(\alpha)$, and hence $\beta < \beta_f(\alpha)$ implies $G_\beta(-u^-) > 0$, which contradicts \eqref{eq:HG=0}.
	
	\ref{prop:aux5:2} Let $\beta > \beta^f(\alpha)$ and $u^+ \not\equiv 0$. Suppose, by contradiction, that $H_\alpha(u^+) \leq 0$. Since $u$ is a solution of \eqref{eq:D}, we have
	\begin{equation}
	\label{eq:HG=02}
	\left<E_{\alpha,\beta}'(u),u^+\right> =  H_\alpha(u^+) + G_\beta(u^+) - \langle f,u^+\rangle = 0.
	\end{equation}
	If $H_\alpha(u^+) < 0$, then we get a contradiction to Proposition \ref{prop:aux4}. If $H_\alpha(u^+) = 0$, then $u^+ \in \mathcal{B}^-(\alpha)$, and hence $\beta > \beta_f(\alpha)$ implies $G_\beta(u^+) < 0$, which contradicts \eqref{eq:HG=02}. The proof is complete.
\end{proof}

\smallskip
\bigskip
\noindent
\textbf{Acknowledgments.}
V.~Bobkov was supported by the grant 18-03253S of the Grant Agency of the Czech Republic and by the project LO1506 of the Czech Ministry of Education, Youth and Sports.
M.~Tanaka was supported by JSPS KAKENHI Grant Number 15K17577.
The authors would like to thank the anonymous referee for valuable remarks and suggestions which helped to improve the manuscript.



\addcontentsline{toc}{section}{\refname}
\small
\begin{thebibliography}{99}
	
	\bibitem{altar}
	Alama, S., \& Tarantello, G. (1996). Elliptic problems with nonlinearities indefinite in sign. Journal of Functional Analysis, 141(1), 159-215.
	\href{http://dx.doi.org/10.1006/jfan.1996.0125}{DOI:10.1006/jfan.1996.0125}
	
	\bibitem{Alleg}
	Allegretto, W., \& Huang, Y. (1998). A Picone's identity for the $p$-Laplacian and applications. Nonlinear Analysis: Theory, Methods \& Applications, 32(7), 819-830. 
	\href{http://dx.doi.org/10.1016/S0362-546X(97)00530-0}{DOI:10.1016/S0362-546X(97)00530-0}
	
	\bibitem{abc}
	Ambrosetti, A., Brezis, H., \& Cerami, G. (1994). Combined effects of concave and convex nonlinearities in some elliptic problems. Journal of Functional Analysis, 122(2), 519-543.
	\href{http://dx.doi.org/10.1006/jfan.1994.1078}{DOI:10.1006/jfan.1994.1078}
	
	\bibitem{anane1987}
	Anane, A. (1987).
	Simplicit\'e et isolation de la premiere valeur propre du $p$-laplacien avec poids.
	Comptes Rendus de l'Acad\'emie des Sciences-Series I-Mathematics,
	305(16), 725-728.
	\url{http://gallica.bnf.fr/ark:/12148/bpt6k57447681/f27}
	
	\bibitem{averna}
	Averna, D., Motreanu, D., \& Tornatore, E. (2016). Existence and asymptotic properties for quasilinear elliptic equations with gradient dependence. Applied Mathematics Letters, 61, 102-107.
	\href{http://dx.doi.org/10.1016/j.aml.2016.05.009}{DOI:10.1016/j.aml.2016.05.009}
	
    \bibitem{BobkovTanaka2015}
	Bobkov, V., \& Tanaka, M. (2015). On positive solutions for $(p,q)$-Laplace equations with two parameters. Calculus of Variations and Partial Differential Equations, 54(3), 3277-3301.
	\href{http://dx.doi.org/10.1007/s00526-015-0903-5}{DOI:10.1007/s00526-015-0903-5} 
	
	\bibitem{BobkovTanaka2016}
	Bobkov, V., \& Tanaka, M. (2016). On sign-changing solutions for $(p,q)$-Laplace equations with two parameters. Advances in Nonlinear Analysis.
	\href{http://dx.doi.org/10.1515/anona-2016-0172}{DOI:10.1515/anona-2016-0172}
	
	\bibitem{BobkovTanaka2017}
	Bobkov, V., \& Tanaka, M. (2018). Remarks on minimizers for $(p, q)$-Laplace equations with two parameters. 
	Communications on Pure and Applied Analysis, 17(3), 1219-1253.
	\href{http://dx.doi.org/10.3934/cpaa.2018059}{DOI:10.3934/cpaa.2018059}
	
	\bibitem{chang}
	Chang, K.~C. (1993). Infinite Dimensional Morse Theory and Multiple Solution Problems. Birkh\"auser.
	\href{http://dx.doi.org/10.1007/978-1-4612-0385-8}{DOI:10.1007/978-1-4612-0385-8}

	\bibitem{chaves}
	Chaves, M. F., Ercole, G., \& Miyagaki, O. H. (2015). Existence of a nontrivial solution for the $(p, q)$-Laplacian in $\mathbb{R}^N$ without the Ambrosetti-Rabinowitz condition. Nonlinear Analysis: Theory, Methods \& Applications, 114, 133-141.
	\href{https://doi.org/10.1016/j.na.2014.11.010}{\nolinkurl{DOI:10.1016/j.na.2014.11.010}}
		
	\bibitem{CP}
	Cl\'ement, P., \& Peletier, L.~A. (1979). An anti-maximum principle for second-order elliptic operators. Journal of Differential Equations, 34(2), 218-229.
	\href{https://doi.org/10.1016/0022-0396(79)90006-8}{\nolinkurl{DOI:10.1016/0022-0396(79)90006-8}}	
	
	\bibitem{drabek}
	Dr\'abek, P. (2002). Geometry of the energy functional and the Fredholm alternative for the $p$-Laplacian in higher dimensions. Electronic Journal of Differential Equations, Conference 08, 103-120.
	\url{https://ejde.math.txstate.edu/conf-proc/08/d1/drabek.pdf}
	
	\bibitem{DGTU}
	Dr\'abek, P., Girg, P., Tak\'a\v{c}, P., \& Ulm, M. (2004). The Fredholm alternative for the $p$-Laplacian: bifurcation from infinity, existence and multiplicity. Indiana University Mathematics Journal, 53(2), 433-482.
	\url{http://www.jstor.org/stable/24903516}

	\bibitem{DR}
	Dr\'abek, P., \& Robinson, S.~B. (1999). Resonance problems for the $p$-Laplacian. Journal of Functional Analysis, 169(1), 189-200.
	\href{http://dx.doi.org/10.1006/jfan.1999.3501}{DOI:10.1006/jfan.1999.3501}

	\bibitem{dug}
	Dugundji, J. (1951). An extension of Tietze's theorem. Pacific Journal of Mathematics, 1(3), 353-367.
	\url{https://projecteuclid.org/euclid.pjm/1103052106}

	\bibitem{ilyas}
	Il'yasov, Y. (2001). On positive solutions of indefinite elliptic equations. Comptes Rendus de l'Acad\'emie des Sciences-Series I-Mathematics, 333(6), 533-538.
	\href{http://dx.doi.org/10.1016/S0764-4442(01)01924-3}{DOI:10.1016/S0764-4442(01)01924-3} 
	
	\bibitem{ilfunc}
	Il'yasov, Y.~S. (2007). Bifurcation calculus by the extended functional method. Functional Analysis and Its Applications, 41(1), 18-30.
	\href{http://dx.doi.org/10.1007/s10688-007-0002-2}{DOI:10.1007/s10688-007-0002-2}	

	\bibitem{flip}
	Filippakis, M.~E., \& Papageorgiou, N.~S. (2018). Resonant $(p,q)$-equations with Robin boundary condition. Electronic Journal of Differential Equations, 2018(1), 1-24.
	\url{https://ejde.math.txstate.edu/Volumes/2018/01/filippakis.pdf}

	\bibitem{FGTT}
	Fleckinger, J., Gossez, J.-P., Tak\'a\v{c}, P., \& de Th\'elin, F. (1995). Existence, nonexistence et principe de l'antimaximum pour le $p$-laplacien. Comptes rendus de l'Acad\'emie des sciences. S\'erie 1, Math\'ematique, 321(6), 731-734.
	\url{http://gallica.bnf.fr/ark:/12148/bpt6k62037127/f81}
	
	\bibitem{takac}
	Fleckinger-Pell\'e J., \& Tak\'a\v{c}, P.
	An improved Poincar\'e inequality and the $p$-Laplacian at resonance for $p>2$.
	Advances in Differential Equations, 7(8), 951-971.
	\url{http://projecteuclid.org/euclid.ade/1356651685}

	\bibitem{FNSS}
	Fu\v{c}\'ik, S., Ne\v{c}as, J., Sou\v{c}ek, J., \& Sou\v{c}ek, V. (2006). Spectral analysis of nonlinear operators (Vol. 346). Springer.
	\href{https://doi.org/10.1007/BFb0059360}{\nolinkurl{DOI:10.1007/BFb0059360}}
	
	\bibitem{Lieberman}
	Lieberman, G.~M. (1988). Boundary regularity for solutions of degenerate elliptic equations. Nonlinear Analysis: Theory, Methods \& Applications, 12(11), 1203-1219. 
	\href{https://doi.org/10.1016/0362-546X(88)90053-3}{\nolinkurl{DOI:10.1016/0362-546X(88)90053-3}}
	
	\bibitem{L}
	Lieberman, G.~M. (1991). The natural generalizationj of the natural conditions of Ladyzhenskaya and Ural'tseva for elliptic equations. Communications in Partial Differential Equations, 16(2-3), 311-361.
	\href{https://doi.org/10.1080/03605309108820761}{\nolinkurl{DOI:10.1080/03605309108820761}}
	
	\bibitem{marcomasconi}
	Marano, S., \& Mosconi, S. (2017). Some recent results on the Dirichlet problem for $(p,q)$-Laplace equations. 
	Discrete and Continuous Dynamical Systems - Series S, 11(2), 279-291.
	\href{https://doi.org/10.3934/dcdss.2018015}{\nolinkurl{DOI:10.3934/dcdss.2018015}}

	\bibitem{MMT}
	Miyajima, S., Motreanu, D., \& Tanaka, M. (2012). Multiple existence results of solutions for the Neumann problems via super-and sub-solutions. Journal of Functional Analysis, 262(4), 1921-1953.
	\href{https://doi.org/10.1016/j.jfa.2011.11.028}{\nolinkurl{DOI:10.1016/j.jfa.2011.11.028}}
	
	\bibitem{MT}
	Motreanu, D., \& Tanaka, M. (2016). On a positive solution for $(p,q)$-Laplace equation with indefinite weight. Minimax Theory and its Applications, 1, 1-20.
	\url{http://www.heldermann-verlag.de/mta/mta01/mta0001-b.pdf}
	
	\bibitem{puser}
	Pucci, P., \& Serrin, J. B. (2007). The maximum principle (Vol. 73). Springer.
	\href{http://dx.doi.org/10.1007/978-3-7643-8145-5}{DOI:10.1007/978-3-7643-8145-5} 
	
	\bibitem{takac2}
	Tak\'a\v{c}, P. (2002). On the Fredholm alternative for the $p$-Laplacian at the first eigenvalue. Indiana University Mathematics Journal, 51(1), 187-238.
	\href{http://dx.doi.org/10.1512/iumj.2002.51.2156}{DOI:10.1512/iumj.2002.51.2156}

	\bibitem{T-2014}
	Tanaka, M. (2014). Generalized eigenvalue problems for $(p,q)$-Laplacian with indefinite weight. Journal of Mathematical Analysis and Applications, 419(2), 1181-1192.
	\href{http://dx.doi.org/10.1016/j.jmaa.2014.05.044}{DOI:10.1016/j.jmaa.2014.05.044}

\end{thebibliography}
\end{document}